\documentclass{amsart}

\usepackage{amsmath}
\usepackage{amssymb,amsfonts,amsthm}
\usepackage{cite}
\usepackage{graphicx}
\usepackage[colorlinks=true, allcolors=blue]{hyperref}
\usepackage[all]{xy}
\usepackage{color,xcolor}

\usepackage{mathrsfs}
\usepackage[all]{xy}
\usepackage{mathabx} 
\usepackage{mathtools}
\usepackage{mathbbol}

\usepackage{wasysym}
\usepackage{cancel,ulem}

\usepackage{mathtools}
\usepackage{mathbbol}
\usepackage{tikz}
\usetikzlibrary{decorations, decorations.markings} 

\usepackage{ifthen}

\newcommand\abs[1]{\left\lvert#1\right\rvert}
\newcommand\bracket[1]{\left\langle#1\right\rangle}

\newtheorem{theorem}{Theorem}[section]
\newtheorem{lemma}[theorem]{Lemma}
\newtheorem{proposition}[theorem]{Proposition}
\newtheorem{corollary}[theorem]{Corollary}

\newtheorem*{assumption}{Assumption}
\newtheorem*{convention}{Notation Convention}

\newtheorem*{lemma*}{Lemma}

\theoremstyle{definition}
\newtheorem{defn}[theorem]{Definition}

\theoremstyle{remark}
\newtheorem{remark}[theorem]{Remark}
\newtheorem{example}[theorem]{Example}
\newtheorem*{claim}{Claim}

\newtheorem{question}{Question}

\def\C{\mathcal C}

\def\sc{\mathcal{SC}}
\def\sh{\mathcal{SH}}

\def\diam{\mathrm{diam}}

\def\stab{\mathrm{Stab}}

\title{A general construction of simultaneously hyperbolic elements}

\author{Jiaqi Cui}
\address{School of Mathematical Sciences, East China Normal University, Shanghai 200241, China P. R.}
\email{51275500053@stu.ecnu.edu.cn}

\author{Renxing Wan}
\address{School of Mathematical Sciences,  Key Laboratory of MEA (Ministry of Education) \& Shanghai Key Laboratory of PMMP,  East China Normal University, Shanghai 200241, China P. R.}
\email{rxwan@math.ecnu.edu.cn}


\begin{document}

\begin{abstract}
    In this paper, we give an explicit construction of simultaneously hyperbolic elements in a group acting on finitely many Gromov-hyperbolic spaces under the weakest conditions. This essentially generalizes results of Clay-Uyanik in \cite{CU18}, of Genevois in \cite{Gen19}, and of Balasubramanya-Fern\'{o}s in \cite{BF25}. Besides, we show that the set of simultaneously hyperbolic elements has strictly positive density with respect to any proper word metric under the weakest conditions. This recovers many classical counting results, eg. the main result of Wiest in \cite{Wie17}. 
    
    As an important ingredient in the proof of main results, we show that the set of simultaneously contracting elements in a group acting with the contracting property on finitely many metric spaces has strictly positive density with respect to any proper word metric. This generalizes two results of Wan-Xu-Yang in \cite{WXY25} and of Balasubramanya-Fern\'{o}s in \cite{BF25}. 
\end{abstract}
\maketitle

\section{Introduction}

This paper solves the following fundamental problem in geometric group theory:
\begin{question}\label{IntroQue}
    Let $G$ be a group acting non-elliptically and non-horocyclically on finitely many  Gromov-hyperbolic spaces $X_1,\ldots, X_l$. Is there a group element $g\in G$ such that $g$ is simultaneously hyperbolic on each $X_i$?
\end{question}
According to Gromov \cite{Gro87}, an isometric group action on a Gromov-hyperbolic space is non-elliptic and non-horocyclic is equivalent to saying that there exists at least one hyperbolic element under this action. It is easy to see that this condition is necessary in Question \ref{IntroQue}.

Question \ref{IntroQue} has previously been considered by several authors, under additional hypotheses. Clay-Uyanik \cite[Theorem 5.1]{CU18} and Genevois \cite[Proposition 6.68]{Gen19} required the additional assumption that every group element is either elliptic or hyperbolic on each $X_i$. As a contrast, Balasubramanya-Fern\'{o}s \cite[Theorem~C]{BF25}  required an additional assumption that each action $G\curvearrowright X_i$ is of general type; see Theorem \ref{Thm: Classification} for details about classifications of group actions on Gromov-hyperbolic spaces. 

Our first result answers Question \ref{IntroQue} affirmatively without adding any condition. 

\begin{theorem}[Theorem~\ref{simulhypiso}]\label{IntroThm: SH}
     Let $G$ be a group acting non-elliptically and non-horocyclically on finitely many Gromov-hyperbolic spaces $X_1,\ldots, X_l$. Then the set $$\mathcal{SH}(G)=\{g\in G: g \text{ is simultaneously hyperbolic on each } X_i \text{ for } 1\le i\le l\}$$ is not empty. 
\end{theorem}

The subsequent result quantifies the proportion of the set  $\sh(G)$ in the whole group. We note that counting problems for discrete groups have a rich historical background. In recent decades, there has been a growing interest in counting problems with respect to the word metric on groups. For related results, see \cite{PS98, CF10a, Wie17, GTT18, Choi25}.

For a group $G$ with a generating set $S$, we denote by $S^{\le n}$ the set of words which can be written as a product of no more than $n$ letters in $S\cup S^{-1}$. 
\begin{theorem}[Corollary~\ref{PosDes sh}]\label{IntroThm: SH pos den}
    Let $G$ be a group acting non-elliptically and non-horocyclically on finitely many Gromov-hyperbolic spaces $X_1,\ldots, X_l$. Suppose $G$ is finitely generated by $S$. Then there exists a constant $c=c(S)\in (0,1)$ such that $$\frac{\sharp  (S^{\le n}\cap \mathcal{SH}(G))}{\sharp S^{\le n}}>c$$ for any sufficiently large $n$.
\end{theorem}

\begin{remark}
    The main result in \cite{Wie17} obtained by Wiest states that under some additional conditions, the proportion of hyperbolic elements in a group acting non-elliptically and non-horocyclically on a Gromov-hyperbolic space has a positive lower bound with respect to any proper word metric. By taking $l=1$, Theorem \ref{IntroThm: SH pos den} generalizes the main result in \cite{Wie17}.
\end{remark}

The most important ingredient in proofs of Theorems \ref{IntroThm: SH}, \ref{IntroThm: SH pos den} is the following result, which generalizes \cite[Theorem~C]{BF25} in the direction that the spaces do not need to be Gromov-hyperbolic. 

\begin{theorem}[Theorem~\ref{thm: simul contracting elements} + Corollary~\ref{PosDes sc}]\label{IntroThm: SC}
    Suppose that a group $G$ acts isometrically and with the contracting property on finitely many geodesic metric spaces $X_1,\ldots, X_l$ respectively. Then the set $$\mathcal{SC}(G)=\{g\in G: g \text{ is simultaneously contracting on each } X_i \text{ for } 1\le i\le l\}$$ is not empty and contains an infinite independent subset.

    Moreover, if $G$ is finitely generated by $S$, then there exists a constant $c=c(S)\in (0,1)$ such that $$\frac{\sharp  (S^{\le n}\cap \mathcal{SC}(G))}{\sharp S^{\le n}}>c$$ for any sufficiently large $n$.
\end{theorem}

Roughly speaking, a contracting element is a ``hyperbolic-like'' element in a general metric space and an isometric group action has contracting property if there exist at least two weakly-independent contracting elements. We refer the readers to \cite{WXY25} or Subsection \ref{subsec: contracting} for more details about contracting elements and contracting property. 

\begin{remark}
\label{rem: not generic}
\begin{enumerate}
    \item In \cite{WXY25}, Wan-Xu-Yang obtained a similar result in the case $l=2$. Thus, Theorem \ref{IntroThm: SC} is also a generalization of \cite[Proposition 4.2]{WXY25}.
    \item We remark that in general, $\sh(G)$ and $\sc(G)$ are not generic, i.e. the constant $c$ in Theorems~\ref{IntroThm: SH pos den}, \ref{IntroThm: SC} can not be arbitrarily close to $1$. See Example~\ref{Exa: DesGAP} for a counterexample.
    \item In recent years, some researchers have been studying how to convert a group action on a metric space with contracting elements into a group action on a Gromov-hyperbolic space with hyperbolic elements. See \cite{BBFS19,PZS24,Zbi24} for references. However, all these works require additional assumptions on the original group action, such as properness. Since our Theorem \ref{IntroThm: SC} imposes no restrictions on the group action, it cannot be directly deduced from Theorem \ref{IntroThm: SH}.
    \item In \cite{FLM18}, Fern\'{o}s-L\'{e}cureux-Math\'{e}us used a probabilistic method to obtain a special case of Theorem \ref{IntroThm: SC} (cf. \cite[Theorem 1.5]{FLM18}). In \cite{Sis18}, Sisto used a probabilistic method to obtain Theorem \ref{IntroThm: SC} under some additional assumptions on the group actions (cf. \cite[Corollary 1.7]{Sis18}). Also note that \cite[Theorem~C]{BF25} relies on the probabilistic method developed by Maher-Tiozzo \cite{MT18}. To the best knowledge of the authors, we are not sure whether purely probabilistic methods could imply Theorem \ref{IntroThm: SH} or \ref{IntroThm: SC}. In this paper, we obtain both results via combinatorial techniques and hyperbolic geometries. 
\end{enumerate}
    
\end{remark}

Actually, Theorem \ref{IntroThm: SH pos den} can be deduced from the following two results. Since these results may be of independent interest themselves, we state them here.

\begin{lemma}[Lemma~\ref{lem: positive density}]
\label{IntroLem: positive density}
    Let $G$ be a group acting non-elliptically and non-horocyclically on finitely many Gromov-hyperbolic spaces $X_1,\ldots, X_l$. Then there exists a finite set $F\subset G$ with the following property. For any $g\in G$, there exists $f\in F$ such that $gf\in \mathcal{SH}(G)$.
\end{lemma}

The other is a general density criterion which can be extracted from the proof of \cite[Corollary~1]{CW18}. 
\begin{lemma}[Lemma~\ref{lem: PosDes}]
    Let $G$ be a group finitely generated by $S$. Let $E\subset G$ be a subset. Suppose there is a finite subset $F\subset G$ with the following property: For any $g\in G$, there is an $f\in F$ such that $gf\in E$. Then there exists a constant $c=c(S)\in (0,1)$ such that $$\frac{\sharp  (S^{\le n}\cap E)}{\sharp S^{\le n}}>c$$ for any sufficiently large $n$.
\end{lemma}

\subsection{Applications}

As an immediate consequence of Theorems \ref{IntroThm: SH} and \ref{IntroThm: SC}, we have
\begin{corollary}\label{IntroCor1}
    Let $G_0,G_1,\ldots,G_l$ be groups such that one of the following holds:
    \begin{enumerate}
        \item each $G_i$ admits an isometric action with contracting property on a geodesic metric space $X_i$.
        \item each $G_i$ admits an isometric action on a Gromov-hyperbolic space $Y_i$ and contains a hyperbolic element.
    \end{enumerate}
    Then for any collection of isomorphisms $\phi_i: G_0\to G_i$ for $1\le i\le l$, there exists a contracting (resp. hyperbolic) element $g\in G_0$ such that $\phi_1(g),\cdots,\phi_l(g)$ are contracting (resp. hyperbolic) elements on $X_1,\cdots,X_l$ (resp. $Y_1,\cdots,Y_l$).
\end{corollary}

\begin{remark}
    In \cite{Sis18}, with some additional assumptions on the group actions, Sisto obtained a similar result\footnote{We remind the readers that the definition of ``contracting elements'' there (see \cite[Definition~1.2]{Sis18}) is different from that in this paper (see Subsection \ref{subsec: contracting}). When the weak path system contains all finite geodesics and the group action is proper, these two notions coincide.} via a probabilistic method (cf. \cite[Corollary 1.7(1)]{Sis18}).  
\end{remark}

Another application of Theorem \ref{IntroThm: SC} is to improve the main result of Dey-Liu in \cite{DL25}. One typical example of a group action with contracting property is a proper group action on a CAT(0) space with rank-1 isometries; see Example \ref{exam: contracting} for more examples of group actions with contracting property. 

Suppose that $G$ is a group acting properly on two proper CAT(0) spaces $X_1,X_2$. The main result of Dey-Liu, i.e. \cite[Theorem 1.2]{DL25} says that if there exists $g\in G$ such that $g$ is a rank-1 isometry on both $X_1$ and $X_2$ and if the diagonal action of $G$ on $X_1\times X_2$ is convex co-compact, then up to rescaling, the two actions $G\curvearrowright X_1$ and $G\curvearrowright X_2$ have the same marked length spectrum.

With the help of Theorem \ref{IntroThm: SC}, \cite[Theorem 1.2]{DL25} can be improved to the following result.
\begin{theorem}
    Let $G$ be a group acting properly on two proper CAT(0) spaces $X_1,X_2$ with rank-1 isometries. If the diagonal action of $G$ on $X_1\times X_2$ is convex co-compact, then up to rescaling, the two actions $G\curvearrowright X_1$ and $G\curvearrowright X_2$ have the same marked length spectrum.
\end{theorem}

\paragraph{\textbf{Structure of the paper}} 
The paper is organized as follows. Section \ref{sec: SC} consists of three subsections. In Subsection~\ref{subsec: contracting}, we introduce preliminary background on contracting elements. In Subsection \ref{subsec: SC}, we utilize some combinatorial techniques to generalize the useful Extension Lemma established by Yang in \cite{Yan19} and then prove the existence of simultaneously contracting elements. In Subsection \ref{subsec: positive density sc}, we follow the strategy of Cumplido-Wiest in \cite{CW18} to establish the positive density of simultaneously contracting elements, which completes the proof of Theorem \ref{IntroThm: SC}. Section \ref{sec: SH} follows a similar structure. Subsection \ref{Subsec: Preliminary} reviews preliminaries  about Gromov-hyperbolic spaces and quasimorphisms. In Subsection~\ref{subsec: SH}, we apply Theorem \ref{IntroThm: SC} together with Busemann quasimorphisms constructed in \cite{CCMT15} to prove Theorem~\ref{IntroThm: SH}. The positive density of simultaneously hyperbolic elements, i.e. Theorem \ref{IntroThm: SH pos den} is obtained in Subsection \ref{subsec: PosiDens}.

\subsection*{Acknowledgments} 
We are grateful to Prof. Wenyuan Yang for many helpful suggestions on the first draft. We also thank the anonymous referees for useful comments. R. W. is supported by NSFC No.12471065 \& 12326601 and in part by Science and Technology Commission of Shanghai Municipality (No. 22DZ2229014).

\section{Constructions of simultaneously contracting elements}\label{sec: SC}

The goal of this section is to prove Theorem \ref{IntroThm: SC}, which also serves as an important ingredient in the proofs of Theorems \ref{IntroThm: SH}, \ref{IntroThm: SH pos den}. 

\subsection{Preliminaries on contracting elements}\label{subsec: contracting}

For two metric spaces $X$ and $Y$, a map $\phi: X\to Y$ is  a \textit{quasi-isometric embedding} if there exist two constants $K\ge 1,\varepsilon\ge 0$ such that
$$\frac{1}{K}d_X(x_1,x_2)-\varepsilon\le d_Y(\phi(x_1),\phi(x_2))\le Kd_X(x_1,x_2)+\varepsilon$$
for any $x_1,x_2\in X$. When $K=1,\varepsilon=0$, the above inequality becomes an equality and gives the definition of an isometric embedding. The image of a (quasi-)isometric embedding from a (possibly bounded or unbounded) interval $I$ to a metric space $X$ is a \textit{(quasi-)geodesic} in $X$.

A metric space $X$ is called \textit{geodesic} if any two points in $X$ can be connected by a geodesic. 

Let $(X,d)$ be a metric space. For a subset $Z\subset X$, we denote by $\diam(Z)$ the diameter of $Z$ and $N_C(Z)$ the closed $C$-neighborhood of $Z$ for a constant $C\ge 0$.
For a closed subset $Y\subset X$, we define the \textit{closest point projection} $\pi_Y: X\to Y$ as $$\forall x\in X, \quad \pi_Y(x):=\{y\in Y\mid d(x,y)=d(x,Y)\}.$$
For a subset $B\subset X$, $\pi_Y(B):=\bigcup_{x\in B}\pi_Y(x)$.

\begin{defn}[Contracting Subset]\label{DEF: Contracting Subset}
Let $(X,d)$ be a geodesic metric space. A closed subset $Y\subseteq X$ is called $C$-\textit{contracting} for $C\geq 0$ if for any geodesic (segment) $\alpha$ in $X$ with $\alpha\cap N_C(Y)=\emptyset$, we have $\diam(\pi_Y(\alpha))\leq C$. 
$Y$ is called a \textit{contracting subset} if there exists $C\geq 0$ such that $Y$ is $C$-contracting.
\end{defn}

See Figure \ref{fig: C-contracting} for an illustration of a $C$-contracting set.

\begin{figure}[h]
    \centering
    
\tikzset{every picture/.style={line width=0.75pt}} 

\begin{tikzpicture}[x=0.75pt,y=0.75pt,yscale=-1,xscale=1,scale=0.8]

\draw    (149.33,225) .. controls (189.33,195) and (215.33,162) .. (322.33,158) .. controls (429.33,154) and (503.33,202) .. (533.33,232) ;
\draw  [fill={rgb, 255:red, 155; green, 155; blue, 155 }  ,fill opacity=0.5 ] (514.33,161) .. controls (541.33,176) and (590.33,238) .. (570.33,258) .. controls (550.33,278) and (520.33,274) .. (494.33,254) .. controls (468.33,234) and (380.33,201) .. (319.33,209) .. controls (258.33,217) and (248.33,217) .. (216.33,240) .. controls (184.33,263) and (119.33,273) .. (114.33,236) .. controls (109.33,199) and (166.33,158) .. (192.33,143) .. controls (218.33,128) and (233.33,116) .. (323.33,116) .. controls (413.33,116) and (487.33,146) .. (514.33,161) -- cycle ;
\draw [color={rgb, 255:red, 80; green, 227; blue, 194 }  ,draw opacity=1 ]   (182.33,56) .. controls (206.33,75) and (275.33,87) .. (331.33,86) .. controls (387.33,85) and (466.33,83) .. (491.33,58) ;
\draw  [dash pattern={on 0.84pt off 2.51pt}]  (182.33,56) .. controls (244.33,94) and (290.33,109) .. (292.33,159) ;
\draw  [dash pattern={on 0.84pt off 2.51pt}]  (491.33,58) .. controls (434.33,97) and (388.33,89) .. (358.33,159) ;
\draw [color={rgb, 255:red, 208; green, 2; blue, 27 }  ,draw opacity=1 ]   (292.33,159) .. controls (359.33,156) and (300.33,157) .. (358.33,159) ;

\draw (193,205) node [anchor=north west][inner sep=0.75pt]   [align=left] {Y};
\draw (533,158) node [anchor=north west][inner sep=0.75pt]   [align=left] {$\displaystyle N_{C}$(Y)};
\draw (330,65) node [anchor=north west][inner sep=0.75pt]   [align=left] {$\displaystyle \alpha $};
\draw (307,132) node [anchor=north west][inner sep=0.75pt]   [align=left] {$\displaystyle \pi _{Y}$($\displaystyle \alpha $)};
\draw (309,163) node [anchor=north west][inner sep=0.75pt]   [align=left] {$\displaystyle \leq C$};

\end{tikzpicture}
    \caption{$Y$ is $C$-contracting.}
    \label{fig: C-contracting}
\end{figure}

Two subsets $Y,Z \subseteq X$ have \textit{$\mathcal{R}$-bounded intersections} for a function $\mathcal{R}: [0,+\infty) \rightarrow [0,+\infty)$ if $\diam(N_r(Y)\cap N_r(Z)) \leq \mathcal{R}(r)$, for all $r \geq 0$.

Let $G$ be a group acting isometrically on a geodesic metric space $(X,d)$. Let $o\in X$ be a basepoint. An element $h \in G$ is called a \textit{contracting element} if the orbit $\bracket{h}\cdot o$ is a contracting subset in $X$ and the map $\mathbb{Z} \rightarrow X$, $n \mapsto h^no$ is a quasi-isometric embedding. One significant source of contracting elements arises from hyperbolic elements in groups that act isometrically on Gromov-hyperbolic spaces.

Suppose $g,h \in G$ are two contracting elements. $g$ and $h$ are called \textit{weakly-independent} if $\bracket{g} \cdot o$ and $\bracket{h} \cdot o$ have $\mathcal{R}$-bounded intersections for some $\mathcal{R} \colon [0,+\infty) \rightarrow [0,+\infty)$. The action $G\curvearrowright X$ is said to have \textit{contracting property} if there exist two weakly-independent contracting elements in $G$. We refer the readers to \cite{WXY25} for more details about contracting subsets and contracting elements.

\begin{example}\cite[Example 2.32]{WXY25}\label{exam: contracting}
     The following provides abundant examples of group actions with contracting property. 
     \begin{enumerate}
         \item A non-elementary group action (i.e. the limit set has at least 3 points) on a Gromov-hyperbolic space.
         \item A group $G$ with non-trivial Floyd boundary (for example, a relatively hyperbolic group) acts on its Cayley graph with respect to a finite generating set $S$.
         \item A non-elementary group action on a $\rm{CAT}(0)$-space with rank-$1$ elements.
         \item A $Gr^\prime(1/6)$-labeled graphical small cancellation group $G$ with finite components labeled by a finite set $S$ acts on the Cayley graph with respect to the generating set $S$.
         \item Any non-virtually cyclic group acts properly on a geodesic metric space with one contracting element.
         \item The mapping class group of a hyperbolic surface $\mathrm{Mod}(\Sigma)$ acts on the Teichm\"{u}ller space $Teich(\Sigma)$ equipped with Teichm\"{u}ller metric, or on the curve complex $\C(\Sigma)$. 
     \end{enumerate}
\end{example}

\subsection{Existence of simultaneously contracting elements}\label{subsec: SC}

\begin{defn}
    Suppose that a group $G$ acts isometrically and with the contracting property on finitely many geodesic metric spaces $X_1,\ldots, X_l$ respectively. Then we define $$\mathcal{SC}(G)\coloneqq\{g\in G: g \text{ is simultaneously contracting on each } X_i \text{ for } 1\le i\le l\}.$$

    Let $A$ be a subset of $\{1,\ldots,l\}$. For convenience, we denote $$\mathcal{SC}(G; A)\coloneqq\{g\in G: g \text{ is simultaneously contracting on each } X_i \text{ for } i\in A\}.$$ Without ambiguity, we briefly write $\sc(G)$ to mean $\sc(G;\{1,\ldots,l\})$ at times.

    A subset $F\subset \sc(G;A)$ is called \textit{independent} if any pair of distinct elements in $F$ are weakly-independent on each $X_i$ for $i\in A$. 
\end{defn}

In this subsection, our goal is to prove the following result, which is the main part of Theorem \ref{IntroThm: SC}.
\begin{theorem}
\label{thm: simul contracting elements}
    Suppose that a group $G$ acts isometrically and with the contracting property on finitely many geodesic metric spaces $X_1,\ldots, X_l$ respectively. Then the set $\mathcal{SC}(G)\ne \varnothing$. Moreover, $\mathcal{SC}(G)$ contains an infinite independent subset.
\end{theorem}

The most important tool to prove Theorem \ref{thm: simul contracting elements} is the following Extension Lemma.

\begin{lemma}[Extension Lemma]\cite[Corollary~3.4(3)]{WXY25}
\label{extensionlemma}
    Suppose that $G$ acts isometrically on a metric space $X$ with a basepoint $o\in X$. Then for any three pairwise weakly-independent contracting elements $h_1,h_2,h_3\in G$, there exists a constant $D>0$ with the following property. Fix any $F = \{f_1, f_2, f_3\}$ with $f_i\in \langle h_i\rangle $ and $ d(o,f_io)>D$. Then, for any $g\in G$ satisfying $d(o,go)>D$, we can choose $f\in F$ so that $gf,fg$ are contracting elements on $X$.
\end{lemma}

Based on the above Extension Lemma, we introduce a combinatorial method which will be used frequently in our constructions.

\begin{lemma}
     Suppose that a group $G$ acts isometrically and with the contracting property on geodesic metric spaces $X_1,\cdots,X_l$ respectively. Fix a basepoint $o_k\in X_k$ for $1\le k\le l$. Suppose that there is an independent subset $F=\{f_1,\cdots,f_{s}\}\subset\sc(G;\{1,\cdots,l\})$ for an $s>2l$. Then there exists a constant $D>0$ with the following property. Suppose $h_1,\cdots,h_t\in G$ satisfy that $d(o_k,h_jo_k)>D$ for any $1\leq k\leq l$, $1\leq j\leq t$. Then there are $t(s-2l)$ elements of the form $f_ih_j$ lying in $\sc(G;\{1,\cdots,l\})$ where $1\leq i\leq s,1\leq j\leq t$.
\end{lemma}    

\begin{proof}
     By Lemma~\ref{extensionlemma}, let $D>0$ be the constant satisfying the following property. For any $1\leq k\leq l$, for any three elements $f_p,f_q,f_r\subset\{f_{1},\cdots,f_{s}\}$ and any three integers $M_p,M_q,M_r$ satisfying $d(o_k,f_i^{M_i}o_k)>D$ for $i=p,q,r$, and any $h\in G$ satisfying $d(o_k,ho_k)>D$, there is an $f\in\{f_p^{M_p},f_q^{M_q},f_r^{M_r}\}$ such that $fh,hf$ are contracting elements on $X_k$. Up to taking higher powers, we assume that $d(o_k,f_io_k)>D$ for any $1\leq k\leq l$, $1\leq i\leq s$.

     Now, suppose $h_1,\cdots,h_t\in G$ satisfy that $d(o_k,h_jo_k)>D$ for any $1\leq k\leq l$, $1\leq j\leq t$. By Lemma~\ref{extensionlemma}, there is an $f\in\{f_1,f_2,f_3\}$ such that $fh_1,h_1f$ are contracting on $X_1$. Then substitute $f$ by $f_4$ in $\{f_1,f_2,f_3\}$. By Lemma~\ref{extensionlemma}, there is an $f^\prime\in\{f_1,f_2,f_3,f_4\}-\{f\}$ such that $f^\prime h_1,h_1f^\prime$ are contracting on $X_1$. Then substitute $f^\prime$ by $f_5$ in $\{f_1,f_2,f_3,f_4\}-\{f\}$. Repeating this process, we get $s-2$ elements $\{f_{i_1},\cdots,f_{i_{s-2}}\}\subset F$ such that $f_{i_p}h_1$ is contracting on $X_1$ for any $1\leq p\leq s-2$. Repeating this process for $\{f_{i_1},\cdots,f_{i_{s-2}}\}$ on $X_2$, we get $s-4$ elements $\{ f_{j_1},\cdots,f_{j_{s-4}} \}\subset \{f_{i_1},\cdots,f_{i_{s-2}}\}$ such that $f_{j_q}h_1$ is contracting on $X_2$ for any $1\leq q\leq s-4$. Notice that $f_{j_q}h_1$ is also contracting on $X_1$ for any $1\leq q\leq s-4$. Repeating this process until $X_l$, we get a subset $F^\prime\subset F$ with $s-2l$ elements such that for any $f\in F^\prime$, $fh_1,h_1f\in \sc(G;\{1,\cdots,l\})$. Do this process for $h_2,\cdots,h_t$, respectively. We get $t(s-2l)$ elements of the form $f_ih_j$ where $1\leq i\leq s,1\leq j\leq t$ such that they are all contracting on $X_1,\cdots,X_l$.
\end{proof}
     
\begin{convention}[SC construction]\label{SC construction}
     We call the above process {\hyperref[SC construction]{SC construction} for $h_1,\cdots,h_t$ on $X_1,\cdots,X_l$ with $f_1,\cdots,f_s$}.
\end{convention}

We will prove $\sc(G)\ne \varnothing$ by induction on $l$. First, we make a overall assumption for convenience.

\begin{assumption}[*]\label{assumption}
    Suppose that a group $G$ acts isometrically and with the contracting property on geodesic metric spaces $X_1,\cdots,X_l$ respectively. We assume that there are two independent subsets $F=\{f_1,\cdots,f_{2s}\}\subset\sc(G;\{1,\cdots,l-1\})$ and $T=\{g_1,\cdots,g_{2s}\}\subset\sc(G;\{2,\cdots,l\})$ for a fixed $s>2l+1$. We fix a basepoint $o_k\in X_k$ for each $1\le k\le l$. 
\end{assumption}

Together with Lemma \ref{extensionlemma}, one immediately gets that
\begin{lemma*}[*]\label{lemma}
        Under \hyperref[assumption]{Assumption (*)}, there exist a uniform constant $D>0$ and sufficiently large integers $M_1,\cdots,M_{2s}, N_1,\cdots, N_{2s}$ such that the following properties hold: 
    \begin{enumerate}
        \item[(P1)] For each $1\leq k\leq l-1$, any $h\in G$ with $d(o_k,ho_k)>D$ and any three integers $\{p,q,r\}\subset \{1,\cdots,2s\}$, there exists $f\in \{f_p^{M_p},f_q^{M_q},f_r^{M_r}\}$ such that $fh,hf$ are contracting elements on $X_k$.
        \item[(P2)] For each $2\leq k\leq l$, any $h\in G$ with $d(o_k,ho_k)>D$ and any three integers $\{p,q,r\}\subset \{1,\cdots,2s\}$, there exists $g\in \{g_p^{N_p},g_q^{N_q},g_r^{N_r}\}$ such that $gh,hg$ are contracting elements on $X_k$.
        \item[(P3)]\label{p1} (\textbf{Distance Estimates}) For any $1\le k_1\le l-1, 2\le k_2\le l, 2\le k_3\le l-1$, one has
        \begin{equation*}
            d(o_{k_1},f_1^{M_1}o_{k_1})-2(2s-1)D>\cdots > d(o_{k_1},f_i^{M_i}o_{k_1})-2(2s-i)D>\cdots >d(o_{k_1},f_{2s}^{M_{2s}}o_{k_1})>3D,
        \end{equation*}
        \begin{equation*}
            d(o_{k_2},g_1^{N_1}o_{k_2})-2(2s-1)D>\cdots > d(o_{k_2},g_i^{N_i}o_{k_2})-2(2s-i)D>\cdots >d(o_{k_2},g_{2s}^{N_{2s}}o_{k_2})>3D,
        \end{equation*}
        \begin{equation*}
            d(o_{k_3},g_{2s}^{N_{2s}}o_{k_3})-2D>2d(o_{k_3},f_1^{M_1}o_{k_3}).
        \end{equation*}
    \end{enumerate}
\end{lemma*}

For convenience, we omit the exponent superscript and use $f_i$ (resp. $g_j$) to represent $f_i^{M_i}$ (resp. $g_j^{N_j}$) for $1\le i,j\le 2s$ in the above lemma.

Now we turn to discuss the following cases according to the actions of $\{g_1,\cdots,g_{2s}\}$ (resp. $\{f_1,\cdots,f_{2s}\}$) on $X_1$ (resp. $X_l$).

\begin{lemma}
\label{f>g>}
    Under \hyperref[assumption]{Assumption (*)} and by \hyperref[lemma]{Lemma (*)}, if $d(o_1,g_jo_1)>D$ for any $1\leq j\leq 2s$ and $d(o_l,f_io_l)>D$ for any $1\leq i\leq 2s$, then some element of the form $f_ig_j$ lies in $\mathcal{SC}(G,\{1,\cdots,l\})$. 
\end{lemma}
\begin{proof}
    Using \hyperref[SC construction]{SC construction} for $g_1,\cdots,g_{2s}$ on $X_1,\cdots,X_{l-1}$ with $f_1,\cdots,f_{2s}$, we get $2s(2s-2(l-1))$ elements of the form $f_ig_j$ where $1\leq i\leq 2s,1\leq j\leq 2s$ such that they are all contracting on $X_1,\cdots,X_{l-1}$.

    Using \hyperref[SC construction]{SC construction} for $f_1,\cdots,f_{2s}$ on $X_l$ with $g_1,\cdots,g_{2s}$, we get $2s(2s-2)$ elements of the form $f_ig_j$ where $1\leq i\leq 2s,1\leq j\leq 2s$ such that they are all contracting on $X_l$.

    Note that the total number of elements of the form $f_ig_j$ where $1\leq i,j\leq 2s$ is at most $(2s)^2$. Since
    $$2s(2s-2(l-1))+2s(2s-2)-(2s)^2=4s^2-4sl>0,$$
    by the Pigeonhole Principle, there must be an element of the form $f_ig_j$ where $1\leq i,j\leq 2s$ such that it is simultaneously contracting on $X_1,\cdots,X_{l-1}$ and $X_l$. 
\end{proof}

\begin{lemma}
\label{f<g<}
    Under \hyperref[assumption]{Assumption (*)} and by \hyperref[lemma]{Lemma (*)}, if $d(o_1,g_jo_1)\leq D$ for any $1\leq j\leq 2s$ and $d(o_l,f_io_l)\leq D$ for any $1\leq i\leq 2s$, then some element of the form $f_if_1g_1g_j$ lies in $\mathcal{SC}(G,\{1,\cdots,l\})$. 
\end{lemma}
\begin{proof}
    For any $1\leq j\leq 2s$, by triangle inequality,  $d(o_1,g_1g_jo_1)\leq 2D$ and then $d(o_1,f_1g_1g_jo_1)\geq d(o_1,f_1o_1)-d(f_1o_1,f_1g_1g_jo_1)>3D-2D=D$. Using \hyperref[SC construction]{SC construction} for $f_1g_1g_1,\cdots,f_1g_1g_{2s}$ on $X_1$ with $f_1,\cdots,f_{2s}$, we get $2s(2s-2)$ elements of the form $f_if_1g_1g_j$ where $1\leq i,j\leq 2s$ such that they are all contracting on $X_1$.

    For any $1\leq i\leq 2s$, by (\hyperref[p1]{P3}) and triangle inequality, one has $$\forall 2\le k\le l-1,\quad d(o_k,f_if_1g_1o_k)\geq -d(o_k,f_if_1o_k)+d(f_if_1o_k,f_if_ig_1o_k)>-2d(o_k,f_1o_k)+d(o_k,g_1o_k)>D.$$ By triangle inequality again,  $d(o_l,f_if_1o_l)\leq 2D$ then $$d(o_l,f_if_1g_1o_l)\geq -d(o_lf_if_1o_l)+d(f_if_1o_l,f_if_1g_1o_l)>-2D+3D=D.$$ Using \hyperref[SC construction]{SC construction} for $f_1f_1g_1,\cdots,f_{2s}f_1g_1$ on $X_2,\cdots,X_l$ with $g_1,\cdots,g_{2s}$, we get $2s(2s-2(l-1))$ elements of the form $f_if_1g_1g_j$ where $1\leq i,j\leq 2s$ such that they are all contracting on $X_2,\cdots,X_l$.

    Note that the total number of elements of the form $f_if_1g_1g_j$ where $1\leq i,j\leq 2s$ is at most $(2s)^2$. Since
    $$2s(2s-2)+2s(2s-2(l-1))-(2s)^2=4s^2-(8l-4)s>0,$$
    by the Pigeonhole Principle, there must be an element of the form $f_if_1g_1g_j$ where $1\leq i,j\leq 2s$ such that it is simultaneously contracting on $X_1$ and $X_2,\cdots,X_l$. 
\end{proof}

\begin{lemma}
\label{f>g<}
    Under \hyperref[assumption]{Assumption (*)} and by \hyperref[lemma]{Lemma (*)}, if $d(o_1,g_jo_1)\leq D$ for any $1\leq j\leq 2s$ and $d(o_l,f_io_l)>D$ for any $1\leq i\leq 2s$, then some element of the form $g_jf_1f_i$ or $g_jg_1f_1f_i$ lies in $\mathcal{SC}(G,\{1,\cdots,l\})$. 
\end{lemma}
\begin{proof}
    Consider $d(o_l,f_1f_io_l)$, $1\leq i\leq 2s$. By the Pigeonhole Principle, there are at least $s$ elements, denoted by $f_{i_1},\cdots,f_{i_s}$ for $i_1<\cdots<i_s$, in $F$ such that either $d(o_l,f_1f_io_l)>D$ for any $i\in\{i_1,\cdots,i_s\}$ or $d(o_l,f_1f_io_l)\leq D$ for any $i\in\{i_1,\cdots,i_s\}$.

    \textbf{Case (1): $d(o_l,f_1f_io_l)>D$ for any $i\in\{i_1,\cdots,i_s\}$.} 
    
    For any $j\in\{1,\cdots,2s\}$, by (\hyperref[p1]{P3}) and triangle inequality, one has $$d(o_1,g_jf_1o_1)\geq -d(o_1,g_jo_1)+d(g_jo_1,g_jf_1o_1)>-D+2D=D$$ and $$\forall 2\le k\le l-1, \quad d(o_k,g_jf_1o_1)\geq d(o_k,g_jo_k)-d(g_jo_k,g_jf_1o_k)>2D>D.$$  
    Using \hyperref[SC construction]{SC construction} for $g_{1}f_1,\cdots,g_{2s}f_1$ on $X_1,\cdots,X_{l-1}$ with $f_{i_1},\cdots,f_{i_s}$, we get $2s(s-2(l-1))$ elements of the form $g_jf_1f_i$ where $i\in\{i_1,\cdots,i_s\},1\leq j\leq 2s$ such that they are all contracting on $X_1,\cdots,X_{l-1}$. 

    Recall that in this case $d(o_l,f_1f_io_l)>D$ for any $i\in\{i_1,\cdots,i_s\}$. Using \hyperref[SC construction]{SC construction} for $f_1f_{i_1},\cdots,f_1f_{i_s}$ on $X_l$ with $g_1,\cdots,g_{2s}$, we get $s(2s-2)$ elements of the form $g_jf_1f_i$ where $i\in\{i_1,\cdots,i_s\},1\leq j\leq 2s$ such that they are all contracting on $X_l$.

    Note that the total number of elements of the form $f_ig_j$ where $i\in\{i_1,\cdots,i_s\},1\leq j\leq 2s$ is at most $2s^2$. Since
    $$2s(s-2(l-1))+s(2s-2)-2s^2=2s^2-(4l-2)s>0,$$
    by the Pigeonhole Principle, there must be an element of the form $g_jf_1f_i$ where $i\in\{i_1,\cdots,i_s\},1\leq j\leq 2s$ such that it is simultaneously contracting on $X_1,\cdots,X_{l-1}$ and $X_l$. 

    \textbf{Case (2): $d(o_l,f_1f_io_l)\leq D$ for any $i\in\{i_1,\cdots,i_s\}$.}

    By triangle inequality, $d(o_1,g_jg_1o_1)\leq 2D$, then (\hyperref[p1]{P3}) implies that $$d(o_1,g_jg_1f_1o_1)\geq -d(o_1,g_jg_1o_1)+d(g_jg_1o_1,g_jg_1f_1o_1)>-2D+3D=D.$$ Using \hyperref[SC construction]{SC construction} for $g_1g_1f_1,\cdots,g_{2s}g_1f_1$ on $X_1$ with $f_{i_1},\cdots,f_{i_s}$, we get $2s(s-2)$ elements of the form $g_jg_1f_1f_i$ where $i\in\{i_1,\cdots,i_s\},1\leq j\leq 2s$ such that they are all contracting on $X_1$. 

    By (\hyperref[p1]{P3}) and  traiangle inequality, one has that for any $2\leq k\leq l-1$ and any $i\in\{i_1,\cdots,i_s\}$, $$d(o_k,g_1f_1f_io_k)\geq d(o_k,g_1o_k)-d(g_1o_k,g_1f_1f_io_k)\geq d(o_k,g_1o_k)-2d(o_k,f_1o_k)>D.$$ Recall that in this case $d(o_l,f_1f_io_l)\leq D$ for any $i\in\{i_1,\cdots,i_s\}$. Then $$d(o_l,g_1f_1f_io_l)\geq d(o_l,g_1o_l)-d(g_1o_l,g_1f_1f_io_l)>2D-D=D.$$ Using \hyperref[SC construction]{SC construction} for $g_1f_1f_{i_1},\cdots,g_1f_1f_{i_s}$ on $X_2,\cdots,X_l$ with $g_1,\cdots,g_{2s}$, we get $s(2s-2(l-1))$ elements of the form $g_jg_1f_1f_i$ where $i\in\{i_1,\cdots,i_s\},1\leq j\leq 2s$ such that they are all contracting on $X_2,\cdots,X_l$. 

    Note that the total number of elements of the form $g_jg_1f_1f_i$ where $i\in\{i_1,\cdots,i_s\},1\leq j\leq 2s$ is at most $2s^2$. Since
    $$2s(s-2)+s(2s-2(l-1))-2s^2=2s^2-(2l+2)s>0,$$
    by the Pigeonhole Principle, there must be an element of the form $g_jg_1f_1f_i$ where $i\in\{i_1,\cdots,i_s\},1\leq j\leq 2s$ such that it is simultaneously contracting on $X_1$ and $X_2,\cdots,X_l$. 
\end{proof}

 The following two lemmas can be obtained in a similar way as Lemma~\ref{f>g<}.

\begin{lemma}
\label{f<g>}
    Under \hyperref[assumption]{Assumption (*)} and by \hyperref[lemma]{Lemma (*)}, if $d(o_1,g_jo_1)>D$ for any $1\leq j\leq 2s$ and $d(o_l,f_io_l)\leq D$ for any $1\leq i\leq 2s$, then some element of the form $g_jg_1f_i$ or $g_jg_1f_1f_i$ lies in $\mathcal{SC}(G,\{1,\cdots,l\})$. 
\end{lemma}
\begin{proof}
    Consider $d(o_1,g_jg_1o_1)$, $1\leq i\leq 2s$. By the Pigeonhole Principle, there are at least $s$ elements, denoted by $g_{j_1},\cdots,g_{j_s}$ for $j_1<\cdots<j_s$, in $S$ such that either $d(o_1,g_jg_1o_1)>D$ for any $j\in\{j_1,\cdots,j_s\}$ or $d(o_1,g_jg_1o_1)\leq D$ for any $j\in\{j_1,\cdots,j_s\}$.

    \textbf{Case (1): $d(o_1,g_jg_1o_1)>D$ for any $j\in\{j_1,\cdots,j_s\}$.}

    Using \hyperref[SC construction]{SC construction} for $g_{j_1}g_1,\cdots,g_{j_{s}}g_1$ on $X_1$ with $f_1,\cdots,f_{2s}$, we get $s(2s-2)$ elements of the form $g_jg_1f_i$, $1\leq i\leq 2s,j\in\{j_1,\cdots,j_s\}$ such that they are all contracting on $X_1$. 

    For any $1\leq i\leq 2s$, by (\hyperref[p1]{P3}) and  triangle inequality, one has that $$\forall 2\leq k\leq l-1,\quad d(o_k,g_1f_io_k)\geq d(o_k,g_1o_k)-d(g_1o_k,g_1f_io_k)\geq d(o_k,g_1o_k)-d(o_k,f_1o_k)>D.$$ By triangle inequality again, $$d(o_l,g_1f_io_l)\geq d(o_l,g_1o_l)-d(g_1o_l,g_1f_io_l)>2D-D=D.$$ Using \hyperref[SC construction]{SC construction} for $g_1f_1,\cdots,g_1f_{2s}$ on $X_2,\cdots,X_l$ with $g_{j_1},\cdots,g_{j_s}$, we get $2s(s-2(l-1))$ elements of the form $g_jg_1f_i$ where $1\leq i\leq 2s,j\in\{j_1,\cdots,j_s\}$ such that they are all contracting on $X_2,\cdots,X_l$. 

    Note that the total number of elements of the form $g_jg_1f_i$ where $1\leq i\leq 2s,j\in\{j_1,\cdots,j_s\}$ is at most $2s^2$. Since
    $$s(2s-2)+2s(s-2(l-1))-2s^2=2s^2-(4l-2)s>0,$$
    by the Pigeonhole Principle, there must be an element of the form $g_jg_1f_i$ where $1\leq i\leq 2s,j\in\{j_1,\cdots,j_s\}$ such that it is simultaneously contracting on $X_1$ and $X_2,\cdots,X_l$. 

    \textbf{Case (2): $d(o_1,g_jg_1o_1)\leq D$ for any $j\in\{j_1,\cdots,j_s\}$.} 
    
    For any $j\in\{j_1,\cdots,j_s\}$, by (\hyperref[p1]{P3}) and triangle inequality, one has $$d(o_1,g_jg_1f_1o_1)\geq -d(o_1,g_jg_1o_1)+d(g_jg_1o_1,g_jg_1f_1o_1)>-D+2D=D.$$ Using \hyperref[SC construction]{SC construction} for $g_{j_1}g_1f_1,\cdots,g_{j_s}g_1f_1$ on $X_1$ with $f_1,\cdots,f_{2s}$, we get $s(2s-2)$ elements of the form $g_jg_1f_1f_i$ where $1\leq i\leq 2s,j\in\{j_1,\cdots,j_s\}$ such that they are all contracting on $X_1$. 

    For any $1\leq i\leq 2s$, by (\hyperref[p1]{P3}) and triangle inequality, one has $$\forall 2\le k\le l-1,\quad d(o_k,g_1f_1f_io_k)\geq d(o_k,g_1o_k)-d(g_1o_k,g_1f_1f_io_k)\geq d(o_k,g_1o_k)-2d(o_k,f_1o_k)>D.$$ By triangle inequality again, $d(o_l,f_1f_io_l)\leq 2D$ and then $$d(o_l,g_1f_1f_io_l)\geq d(o_l,g_1o_l)-d(g_1o_l,g_1f_1f_io_l)>3D-2D=D.$$  
    Using \hyperref[SC construction]{SC construction} for $g_1f_1f_1,\cdots,g_1f_1f_{2s}$ on $X_2,\cdots,X_l$ with $g_{j_1},\cdots,g_{j_s}$, we get $2s(s-2(l-1))$ elements of the form $g_jg_1f_1f_i$ where $1\leq i\leq 2s,j\in\{j_1,\cdots,j_s\}$ such that they are all contracting on $X_2,\cdots,X_l$.

    Note that the total number of elements of the form $g_jg_1f_1f_i$ where $1\leq i\leq 2s,j\in\{j_1,\cdots,j_s\}$ is at most $2s^2$. Since
    $$s(2s-2)+2s(s-2(l-1))-2s^2=2s^2-(4l-2)s>0,$$
    by the Pigeonhole Principle, there must be an element of the form $g_jg_1f_1f_i$ where $1\leq i\leq 2s,j\in\{j_1,\cdots,j_s\}$ such that it is simultaneously contracting on $X_1$ and $X_2,\cdots,X_l$.
\end{proof}

The following two lemmas in \cite{WXY25} guarantee the first step of mathematical induction.
\begin{lemma}\cite[Lemma~3.6]{WXY25}
\label{4independent}
    Suppose that $G$ acts isometrically and with the contracting property on a geodesic metric space $X$. For any contracting element $g\in G$ and any three pairwise weakly-independent contracting elements $h_1,h_2,h_3\in G$, there exists $1\le i\le 3$ such that $g$ and $h_i$ are weakly-independent.
\end{lemma}

\begin{lemma}\cite[Lemma~2.30]{WXY25}
\label{lem: conjuagte contracting}
     Suppose that $G$ acts isometrically and with the contracting property on a geodesic metric space $X$. Let $g,h \in G$ be two weakly-independent contracting elements. Then there exists $N > 0$ such that the following holds.
     \begin{enumerate}
         \item For any $n,m \geq N$, $g^nh^m$ is a contracting element.
         \item For any $n,m \geq N$, $\{(h^mg^n)^kh(h^mg^n)^{-k} \mid k \in \mathbb{Z}\}$ is a collection of pairwise weakly-independent contracting elements.
     \end{enumerate}
\end{lemma}

Note that Lemma~\ref{lem: conjuagte contracting} implies that if $G$ acts isometrically on a metric space $X$ with two weakly-independent contracting elements, then there exists an infinite collection of pairwise weakly-independent contracting elements in $G$ on $X$.

The following lemma is the last piece of the puzzle required for the proof of Theorem~\ref{thm: simul contracting elements}. 

\begin{lemma}\label{lem: InfIndSet}
    Suppose that $G$ acts isometrically and with the contracting property on geodesic metric spaces $X_1,\ldots, X_l$ respectively. Suppose that $\sc(G;\{1,\cdots,l\})\neq \varnothing$ and $\sc(G;\{1,\cdots,l-1\})$ contains an infinite independent subset. Then $\sc(G;\{1,\cdots,l\})$ contains an infinite independent subset.
\end{lemma}
\begin{proof}
    By Lemma~\ref{lem: conjuagte contracting}, it suffices to show that $\sc(G;\{1,\cdots,l\})$ contains an independent pair of elements. By assumption, let $f\in\sc(G;\{1,\cdots,l\})$ and $\{g_1,\cdots,g_s\}\subset\sc(G;\{1,\cdots,l-1\})$ be an independent set for $s>2(l-1)$. The following combinatorial construction is parallel to the \hyperref[SC construction]{SC construction} but does not have the same hypotheses.
    
    By Lemma~\ref{4independent}, there is $g\in\{g_1,g_2,g_3\}$ such that $f,g$ are weakly-independent on $X_1$, Suppose $g=g_1$. Then by Lemma~\ref{4independent} again, there is $g^\prime\in\{g_2,g_3,g_4\}$ such that $f,g^\prime$ are weakly-independent on $X_1$, Suppose $g^\prime=g_2$. Repeating this process, we can suppose $f,g_1,\cdots,g_{s-2}\subset\sc(G;\{1\})$ is an independent set. Then do this process on $X_2$. We can suppose $f,g_1,\cdots,g_{s-4}\subset\sc(G;\{2\})$ is an independent set, and thus $f,g_1,\cdots,g_{s-4}\subset\sc(G;\{1,2\})$ is an independent set. Repeating this process until on $X_{l-1}$, we can suppose that $f,g_1,\cdots,g_{s-2(l-1)}\subset\sc(G;\{1,\cdots,l-1\})$ is an independent set. Since $s>2(l-1)$, we obtain that there exists $g\in\sc(G;\{1,\cdots,l-1\})$ such that $f,g$ are weakly-independent on $X_1,\cdots,X_{l-1}$. By Lemma~\ref{lem: conjuagte contracting}, there is an $n$ such that $\{(f^ng^n)^kf(f^ng^n)^{-k}|k\in\mathbb{Z}\}$ is a collection of pairwise weakly-independent contracting elements on $X_1,\cdots,X_{l-1}$. 
    
    Since $G\curvearrowright X_l$ has contracting property, we can find a contracting element $h$ on $X_l$ such that $f,h$ are weakly-independent on $X_l$. Then by Lemma~\ref{lem: conjuagte contracting}, there is an $m$ such that $\{(f^mh^m)^kf(f^mh^m)^{-k}|k\in\mathbb{Z}\}$ is a collection of pairwise weakly-independent contracting elements on $X_l$. 
    
    Denote $\theta_k=(f^ng^n)^kf(f^ng^n)^{-k}$ and $\eta_k=(f^mh^m)^kf(f^mh^m)^{-k}$. Note that $\theta_k,\eta_k$ for $k\in\mathbb{Z}$ are all conjugations of $f$. Hence, they are all simultaneously contracting elements on $X_1,\cdots,X_l$.

    Now fix $\theta_1,\cdots,\theta_s$ and $\eta_1,\cdots,\eta_s$ for $s>2l$. For $\eta_1$, by Lemma~\ref{4independent}, there is $\theta\in\{\theta_1,\theta_2,\theta_3\}$ such that $\eta_1,\theta$ are weakly-independent on $X_1$. Suppose $\theta=\theta_1$. By Lemma~\ref{4independent}, there is $\theta^\prime\in\{\theta_2,\theta_3,\theta_4\}$ such that $\eta_1,\theta^\prime$ are weakly-independent on $X_1$. Suppose $\theta^\prime=\theta_2$. Repeating this process, we can suppose that $\eta_1,\theta_1,\cdots,\theta_{s-2}$ are pairwise weakly-independent on $X_1$. Then do this process on $X_2$. We can suppose $\eta,\theta_1,\cdots,\theta_{s-4}\subset\sc(G;\{2\})$ is an independent set, and then $\eta,\theta_1,\cdots,\theta_{s-4}\subset\sc(G;\{1,2\})$ is an independent set. Repeating this process until on $X_{l-1}$, we can suppose that $\eta,\theta_1,\cdots,\theta_{s-2(l-1)}\subset\sc(G;\{1,\cdots,l-1\})$ is an independent set. Do this process for $\eta_2,\cdots,\eta_s$. We get $s(s-2(l-1))$ pairs of the form $(\theta_i,\eta_j)$ where $1\leq i,j\leq s$ such that each pair gives an independent subset of $\sc(G;\{1,\cdots,l-1\})$.

    For $\theta_1$, by Lemma~\ref{4independent}, we can suppose that $\theta_1,\eta_1,\cdots,\eta_{s-2}$ are pairwise weakly-independent on $X_l$. Do this process for $\theta_2,\cdots,\theta_s$. We get $s(s-2)$ pairs of the form $(\theta_i,\eta_j)$ where $1\leq i,j\leq s$ such that each pair gives an independent subset of $\sc(G;\{l\})$.

    Note that the total number of pairs of the form $(\theta_i,\eta_j)$ where $1\leq i,j\leq s$ is $s^2$. Since
    $$s(s-2(l-1))+s(s-2)-s^2=s^2-2ls>0,$$
    by the Pigeonhole Principle, there must be a pair of the form $(\theta_i,\eta_j)$ $1\leq i,j\leq s$ such that $\theta_i,\eta_j$ are weakly-independent contracting elements on $X_1,\cdots,X_l$.
\end{proof}

\begin{proof}[Proof of Theorem~\ref{thm: simul contracting elements}]
    We will prove by induction.
    For $l=1$, the conclusion follows from Lemma~\ref{lem: conjuagte contracting}. Suppose that the conclusion holds for $l-1$. 
    
    By induction hypothesis, \hyperref[assumption]{Assumption (*)} and then \hyperref[lemma]{Lemma (*)} holds. In particular, we can require that \hyperref[assumption]{Assumption (*)} and \hyperref[lemma]{Lemma (*)} hold for subscript $4s$ instead of $2s$.

    Consider the actions of $g_{1},\cdots,g_{4s}$ on $X_1$. By the Pigeonhole Principle, there are at least $2s$ elements, denoted by $g_{j_1},\cdots,g_{j_{2s}}$, in $\{g_1,\cdots,g_{4s}\}$ such that either $d(o_1,g_{j}o_1)>D$ or $d(o_1,g_{j}o_1)\leq D$ for any $j\in\{j_1,\cdots,j_{2s}\}$. Similarly, there are at least $2s$ elements, denoted by $f_{i_1},\cdots,f_{i_{2s}}$, in $\{f_1,\cdots,f_{4s}\}$ such that either $d(o_l,f_{i}o_l)>D$ or $d(o_l,f_{i}o_l)\leq D$ for any $i\in\{i_1\cdots, i_{2s}\}$. So we have the following four cases in total: 
        \begin{enumerate}
            \item $d(o_l,f_io_l)>D$ for any $i\in\{i_1,\cdots,i_{2s}\}$, $d(o_1,g_jo_1)>D$ for any $j\in\{j_1,\cdots,j_{2s}\}$;
            \item $d(o_l,f_io_l)>D$ for any $i\in\{i_1,\cdots,i_{2s}\}$, $d(o_1,g_jo_1)\leq D$ for any $j\in\{j_1,\cdots,j_{2s}\}$;
            \item $d(o_l,f_io_l)\leq D$ for any $i\in\{i_1,\cdots,i_{2s}\}$, $d(o_1,g_jo_1)>D$ for any $j\in\{j_1,\cdots,j_{2s}\}$;
            \item $d(o_l,f_io_l)\leq D$ for any $i\in\{i_1,\cdots,i_{2s}\}$, $d(o_1,g_jo_1)\leq D$ for any $j\in\{j_1,\cdots,j_{2s}\}$.
        \end{enumerate}

    Cases~(1)~(2)
    ~(3)~(4) satisfy the conditions of Lemmas~\ref{f>g>},~\ref{f>g<},~\ref{f<g>},~\ref{f<g<} respectively by substitute $f_1,\cdots,f_{2s},g_1,\cdots,g_{2s}$ by $f_{i_1},\cdots,f_{i_{2s}},g_{i_1},\cdots,g_{i_{2s}}$. Hence, it follows from Lemmas~\ref{f>g>},~\ref{f>g<},~\ref{f<g>},~\ref{f<g<} that $\sc(G;\{1,\cdots,l\})\neq \varnothing$. The remaining proof is completed by Lemma~\ref{lem: InfIndSet} and induction hypothesis.
\end{proof}

\subsection{Positive density of $\sc(G)$}\label{subsec: positive density sc}

In \cite{CW18}, Cumplido-Wiest used \cite[Theorem 2]{CW18} to show the positive density of pseudo-Anosov elements in mapping class groups \cite[Corollary 1]{CW18}. Here, we first give a lemma which is an analogue of \cite[Theorem 2]{CW18} in our situations. 

\begin{lemma}
\label{generalized extension lemma}
    Suppose that a group $G$ acts isometrically and with the contracting property on finitely many geodesic metric spaces $X_1,\cdots,X_l$ respectively. Then there is a finite set $F\subset \sc(G)$ with the following property. For any $g\in G$, there is an $f\in F$ such that $gf\in \sc(G)$.
\end{lemma}

\begin{proof}
    By Theorem \ref{thm: simul contracting elements}, we can choose an independent set $\{h_1,\cdots,h_{s}\}\subset\sc(G)$ for $s>2l+1$. Fix a basepoint $o_k\in X_k$ for $1\le k\le l$. Let $D$ be the constant given by  \hyperref[SC construction]{SC construction}.

    By Lemma~\ref{lem: conjuagte contracting}, there is an $N\gg0$ such that $h_i^nh_j^m\in\sc(G)$ for any $n,m>N$ and $1\leq i,j\leq s$. We pick $n_i>N$ for any $1\leq i\leq s$ such that 
    $$\min_{1\leq k\leq l}d(o_k,h_i^{n_i}o_k)\ge 2D.$$

    Let $f_i=h_i^{n_i}$ for $1\le i\le s$.
    Denote $L_k=d(o_k,f_1o_k)$ for $1\leq k\leq l$ for short. 

    Then pick $j_0=0<j_1<\cdots<j_{l}$ inductively such that 
    $$d(o_k,(f_1)^{j_r}o_k)>2D+j_{r-1}L_k$$ for any $1\leq k,r\leq l$. 

    Let $F^\prime=\{f_1,\cdots,f_{s}\}$ and 
    $$F=\bigcup_{r=0}^{l} f_1^{j_r}\cdot F^\prime.$$ By construction, $F^\prime\subset F\subset \sc(G)$.

    For any $g\in G$, we aim to find the desired element $f \in F$ according to these values $d(o_1,go_1), \cdots, d(o_l,go_l)$. 
    For any $1\leq k\leq l$, denote $I_{k,0}=[0,D]$, $I_{k,r}=(D+j_{r-1}L_k,D+j_rL_k]$ for $1\leq r\leq l-1$ and $I_{k,l}=(D+j_{l-1}L_k,+\infty)$. Note that $d_k\in \bigsqcup_{r=0}^l I_{k,r}=[0,+\infty)$ for any $1\leq k\leq l$.
    By the Pigeonhole Principle, there exists $0\leq r\leq l$ such that $d(o_k,go_k)\notin I_{k,r}$ for any $1\leq k\leq l$. See Table~\ref{table: classification} for an illustration.

\begin{table}[h]
    \centering
\begin{tabular}{|c|c|c|c|c|c|}
    \hline
    $\in I_{k,r}$ & $r=0$ & $r=1$ & $r=2$ & $\cdots$ & $r=l$\\
    \hline
    $k=1$ && $d_1$ &&&\\
    \hline
    $k=2$ &&&&& $d_2$\\
    \hline
    $\vdots$ &&&&&\\
    \hline
    $k=l$ & $d_l$ &&&&\\
    \hline
\end{tabular}
\caption{Each $d_k\coloneqq d(o_k,go_k)$ for $1\leq k\leq l$ lies in one of the $l+1$ boxes in $k$-th row. By the Pigeonhole Principle, there is an empty column.}
\label{table: classification}
\end{table}

    \begin{claim}
        $d(o_k,gf_1^{j_r}o_k)>D$ for any $1\le k\le l$.
    \end{claim}
    \begin{proof}[Proof of Claim]
        If $r=0$, then it follows from the construction of $I_{k,r}$ that $d(o_k,go_k)>D$ for any $1\le k\le l$. If $r=l$, then $d(o_k,go_k)\le D+j_{l-1}L_k$ for any $1\le k\le l$. By triangle inequality, $$d(o_k,gf_1^{j_l}o_k)\ge d(o_k,f_1^{j_l}o_k)-d(o_k,go_k)> 2D+j_{l-1}L_k-(D+j_{l-1}L_k)=D.$$ 

        Suppose now $1\le r\le l-1$. For those $1\leq k\leq l$ such that $d(o_k,go_k)>D+j_rL_k$, 
        by triangle inequality, we have 
$$d\left( o_k,gf_1^{j_{r}} o_k \right)\geq d\left( o_k,go_k \right)-j_r d\left( o_k,f_1 o_k \right)>D+j_{r}L_k-j_{r}L_k=D.$$ For those $1\leq k\leq l$ such that $d(o_k,go_k)\leq D+j_{r-1}L_k$, 
by triangle inequality, we have

$$d\left( o_k,gf_1^{j_{r}} o_k \right)\geq d\left( o_k,f_1^{j_{r}} o_k \right)-d\left( o_k,go_k \right)>(2D+j_{r-1}L_k)-(D+j_{r-1}L_k)=D.$$ 
    \end{proof}

    By the above Claim, we can apply \hyperref[SC construction]{SC construction} for $gf_1^{j_r}$ on $X_1,\cdots,X_l$ with $F'$ and finally get an $f\in f_1^{j_r}F'$ such that $gf\in \sc(G)$.

\end{proof}

The following lemma can be extracted from the proof of \cite[Corollary~1]{CW18}. For a group $G$ generated by $S$, we use $|g|_S$ to denote the smallest number of letters in $S\cup S^{-1}$ to represent $g$. Then the word metric $d_S$ on $G$ is defined by $d_S(g,h)=|g^{-1}h|_S$. 
\begin{lemma}\label{lem: PosDes}
    Let $G$ be a group finitely generated by $S$. Let $E\subset G$ be a subset. Suppose there is a finite subset $F\subset G$ with the following property: For any $g\in G$, there is an $f\in F$ such that $gf\in E$. Then there exists a constant $c=c(S)\in (0,1)$ such that $$\frac{\sharp  (S^{\le n}\cap E)}{\sharp S^{\le n}}>c$$ for any sufficiently large $n$.
\end{lemma}

\begin{proof}
     Denote $M=\max_{f\in F}\abs{f}_S>0$. Suppose $n>2M$.

    \begin{claim}
        For any $g\in S^{\leq n}$, there exists $h\in S^{\leq n}\cap E$ such that $d_S(g,h)\leq 2M$.
    \end{claim}
    \begin{proof}[Proof of Claim]
        By assumption, there exists $f\in F$ such that $gf\in E$. If $\abs{g}_S\leq M$, then it follows from $d_S\left(g,gf\right)=\abs{f}_S\leq M$ that $\abs{gf}_S\leq\abs{g}_S+\abs{f}_S\leq 2M<n$. Therefore $h=gf$ is as desired.

        If $\abs{g}_S>M$, then we choose a geodesic $\gamma$ from $1$ to $g$ in the Cayley graph $\mathcal{G}(G,S)$ and $g^\prime\in\gamma$ such that $d_{S}\left(g,g^\prime\right)=M$. See Figure~\ref{fig: positive density} for an illustration. By assumption, there exists $f\in F$ such that $g^\prime f\in E$.  Moreover, we have $\abs{g^\prime f}_S\leq\abs{g^\prime}_S+\abs{f}_S= \abs{g}_S-M+M=\abs{g}_S\leq n$. Notice that $d_S\left(g,g^\prime f\right)\leq d_S\left(g,g^\prime\right)+d_S\left(g^\prime,g^\prime f\right)\leq M+\abs{f}_S\leq 2M$. Therefore, $h=g^\prime f$ is as desired.
        \begin{figure}[h]
            \centering
            \begin{tikzpicture}[x=0.75pt,y=0.75pt,yscale=-1,xscale=1,scale=0.6]
            \draw  [fill=gray  ,fill opacity=0.5 ]  (339,221) .. controls (358,205) and (448,189) .. (456,222) .. controls (464,255) and (458,275) .. (469,305) .. controls (480,335) and (335,338) .. (315,308) .. controls (295,278) and (320,237) .. (339,221) -- cycle  ;
            \draw  [dashed]  (251.24,293.07) -- (291.64,69.97) ;
            \draw [shift={(292,68)}, rotate = 100.27] [color={rgb, 255:red, 0; green, 0; blue, 0 }  ][line width=0.75]    (10.93,-3.29) .. controls (6.95,-1.4) and (3.31,-0.3) .. (0,0) .. controls (3.31,0.3) and (6.95,1.4) .. (10.93,3.29)   ;
            \draw    (404,139) -- (251.24,293.07) ;
            \draw    (348,196) -- (429,238) ;
            \draw  [dashed]  (404,139) -- (429,238) ;
            \draw  [dashed] (158.27,84.94) .. controls (202.16,64.29) and (252.75,57.69) .. (303.28,69.45) .. controls (424.38,97.63) and (499.25,220.6) .. (470.51,344.1) .. controls (467.84,355.58) and (464.37,366.65) .. (460.18,377.28) ;  
            \draw   (229.75,64.88) .. controls (253.74,62.33) and (278.52,63.69) .. (303.28,69.45) .. controls (419.47,96.49) and (493.1,210.77) .. (473.51,329.07) ;  
            \fill (251.24,293.07) circle (0.25em);
            \fill (404,139) circle (0.25em);
            \fill (348,196) circle (0.25em);
            \fill (429,238) circle (0.25em);
            \draw (234,294.4) node [anchor=north west][inner sep=0.75pt]    {{\footnotesize $1$}};
            \draw (250,160) node [anchor=north west][inner sep=0.75pt]    {{\footnotesize $n$}};
            \draw (382,127) node [anchor=north west][inner sep=0.75pt]    {{\footnotesize $g$}};
            \draw (327,176.4) node [anchor=north west][inner sep=0.75pt]    {{\footnotesize $g^{\prime }$}};
            \draw (418,240) node [anchor=north west][inner sep=0.75pt]    {{\footnotesize $g^{\prime } f$}};
            \draw (347,150) node [anchor=north west][inner sep=0.75pt]    {{\footnotesize $M$}};
            \draw (354,226) node [anchor=north west][inner sep=0.75pt]    {{\footnotesize $\leq M$}};
            \draw (292,217) node [anchor=north west][inner sep=0.75pt]    {{\footnotesize $\gamma $}};
            \draw (350,290) node [anchor=north west][inner sep=0.75pt]    {{\scriptsize $S^{\leq n}\cap E$}};
        \end{tikzpicture}
        \caption{$|g|_S>M$.}
        \label{fig: positive density}
    \end{figure}
    \end{proof}

    The above Claim shows that
    $$S^{\leq n}\subset\bigcup_{h\in S^{\leq n}\cap E}B_{2M}(h).$$
    Counting their cardinalities gives that
    $$\sharp S^{\leq n}\leq \sharp\left(S^{\leq n}\cap E\right)\cdot \sharp S^{\leq 2M},$$
    i.e.
    $$\frac{\sharp  (S^{\le n}\cap E)}{\sharp S^{\le n}}\geq\frac{1}{\sharp S^{\leq 2M}}.$$

    By setting $c=c(S)=\frac{1}{\sharp S^{\leq 2M}}$, we have
    $$\frac{\sharp  (S^{\le n}\cap E)}{\sharp S^{\le n}}\geq c,$$
    for any $n>2M$.
\end{proof}

\begin{corollary}[Positive density of simultaneously contracting elements]\label{PosDes sc}
    Let $G$ be a group acting isometrically and with the contracting property on finitely many geodesic metric spaces $X_1,\ldots, X_l$ respectively. If $G$ is finitely generated by $S$, then there exists a constant $c=c(S)\in (0,1)$ such that $$\frac{\sharp  (S^{\le n}\cap \mathcal{SC}(G))}{\sharp S^{\le n}}>c$$ for any sufficiently large $n$.
\end{corollary}
\begin{proof}
    Combining Lemma~\ref{generalized extension lemma} and Lemma~\ref{lem: PosDes} completes the proof.
\end{proof}

We will use the same method to prove the positive density of simultaneously hyperbolic elements in subsection~\ref{subsec: PosiDens}.

\section{Constructions of simultaneously hyperbolic elements}\label{sec: SH}

In this section, we answer Question~\ref{IntroQue} and give the positive density of simultaneously hyperbolic elements.

\subsection{Preliminaries}\label{Subsec: Preliminary}

\subsubsection{Gromov-hyperbolic spaces}\label{Subsubsec: HypSpace}


A geodesic metric space $X$ is called \textit{$\delta$-hyperbolic} for some $\delta\ge 0$ if any geodesic triangle in $X$ is \textit{$\delta$-thin}: any side of the triangle is contained in the $\delta$-neighborhood of the union of the other two sides. See Figure~\ref{fig:slim} for an illustration. A \textit{Gromov-hyperbolic space} is a $\delta$-hyperbolic space for some $\delta\geq0$. 

\begin{figure}[h]
    \centering
    \begin{tikzpicture}[x=0.75pt,y=0.75pt,yscale=-1,xscale=1,scale=0.5]
    \draw  [fill={rgb, 255:red, 208; green, 2; blue, 27 }  ,fill opacity=1 ,opacity=0.5] (225,372) .. controls (273,356) and (371,352) .. (421,370) .. controls (471,388) and (527,414) .. (516,451) .. controls (505,488) and (467,486) .. (413,463) .. controls (359,440) and (275,457) .. (227,475) .. controls (179,493) and (151,481) .. (144,447) .. controls (137,413) and (177,388) .. (225,372) -- cycle ;
    \draw  [fill={rgb, 255:red, 155; green, 155; blue, 155 }  ,fill opacity=1 , opacity=0.5] (381,254) .. controls (392,311) and (441,367) .. (487,384) .. controls (533,401) and (541,439) .. (518,470) .. controls (495,501) and (458,497) .. (404,474) .. controls (350,451) and (272,361) .. (261,316) .. controls (250,271) and (247,197) .. (291,190) .. controls (335,183) and (370,197) .. (381,254) -- cycle ;
    \draw    (189,433) .. controls (314,383) and (324,369) .. (321,244) ;
    \draw    (470,442) .. controls (355,394) and (325,374) .. (321,244) ;
    \draw    (189,433) .. controls (343,373) and (337,393) .. (470,442) ;
    \node[xshift=-0.5em] at (189,433) {$y$};
    \node[yshift=0.5em] at (321,244) {$x$};
    \node[xshift=0.5em] at (470,442) {$z$};
    \end{tikzpicture}
    \caption{$[x,y]\subset N_\delta([y,z]\cup[z,x])$.}
    \label{fig:slim}
\end{figure}

Now we recall the definition of the Gromov boundary of a Gromov-hyperbolic space. Let $X$ be a geodesic metric space. Fix a base point $x_0\in X$. Given $x,y\in X$, we define the \textit{Gromov product}
$$(x\cdot y)_{x_0}\coloneqq \frac{1}{2}\left(d\left(x_0,x\right)+d\left(x_0,y\right)-d(x,y)\right).$$

\begin{defn}\cite[Chapter H, Definition 3.12]{BH99}
Let $X$ be a  $\delta$-hyperbolic space with a basepoint $x_0$. A sequence $(x_n)=(x_n)_{n=1}^{\infty}$ in $X$ \textit{converges at infinity} if $(x_i,x_j)_{x_0}\to \infty$ as $i,j\to \infty$.
The Gromov boundary $\partial X$ is defined as the following set, equipped with an appropriate metrizable topology,
$$\partial X:=\{(x_n) \subset X : (x_n)\text{ converges at infinity}\}/\sim$$
where $(x_n)\sim (y_n)$ if $(x_i,y_j)_{x_0}\to\infty$ as $i,j\to \infty$. 
\end{defn}

Let $G$ be a group acting by isometries on a Gromov-hyperbolic space $X$. 
The \textit{limit set} of this action is defined by
$$\Lambda(G)\coloneqq\left\{ [(x_n)]\in \partial X \mid \text{there exists } \{g_n\}\subset G \text{ such that } x_n=g_n\cdot x_0\right\}.$$

Moreover, an isometric action $G\curvearrowright X$ induces an action of $G$ on $\partial X$ by homeomorphisms, given by $g\cdot[(x_n)]\coloneqq[(g\cdot x_n)]$. It is also well-known that $\Lambda(G)\subset \partial X$ is a $G$-invariant closed subset and the cardinality of $\Lambda(G)$ is 0, 1, 2 or $\infty$.

\begin{defn}
    Suppose $G$ acts isometrically on a Gromov-hyperbolic space $X$. Let $g\in G$. If the orbit of $\langle g\rangle$ in $X$ is bounded, then $g$ is called \textit{elliptic}. Otherwise $g$ is called \textit{hyperbolic} (resp. \textit{parabolic}) if $g$ has exactly two (resp. one) fixed points on $\partial X$.
\end{defn}

According to Gromov \cite{Gro87}, all isometric actions on a Gromov-hyperbolic space can be divided into the following five types. 
\begin{theorem}\cite[Theorem 4.2]{ABO19}\label{Thm: Classification}
    Let $G$ be a group acting isometrically on a Gromov-hyperbolic space $X$. Then exactly one of the following conditions holds. 
    \begin{enumerate}
        \item $\abs{\Lambda(G)}=0$. Equivalently, $G$ has bounded orbits. In this case the action of $G$ is called {\rm elliptic}.
        \item $\abs{\Lambda(G)}=1$. Equivalently, $G$ has unbounded orbits and contains no hyperbolic elements. In this case the action of $G$ is called {\rm parabolic} or {\rm horocyclic}. A parabolic action cannot be cobounded and the set of points of $\partial X$ fixed by $G$ coincides with $\Lambda(G)$.
        \item $\abs{\Lambda(G)}=2$. Equivalently, $G$ contains a hyperbolic element and any two hyperbolic elements have the same limit points on $\partial X$. In this case the action of $G$ is called {\rm lineal}.
        \item $\abs{\Lambda(G)}=\infty$. Then $G$ always contains hyperbolic elements. In turn, this case breaks into two subcases.\begin{enumerate}
            \item $G$ fixes a point of $\partial X$. Equivalently, any two hyperbolic elements of $G$ have a common limit point on the boundary. In this case the action of $G$ is called {\rm quasi-parabolic} or {\rm focal}. Orbits of quasi-parabolic actions are always quasi-convex.
            \item $G$ does not fix any point of $\partial X$. Equivalently, $G$ contains infinitely many independent hyperbolic elements. In this case the action of $G$ is said to be {\rm of\ general\ type}.
        \end{enumerate}
    \end{enumerate}
\end{theorem}

\subsubsection{Busemann quasimorphisms}

\begin{defn}\label{Def: QM}
    For a group $G$, a map $\varphi\colon G\rightarrow\mathbb{R}$ is a \textit{quasimorphism} if the \textit{defect} of $\varphi$ defined by
    $$\Delta(\varphi):=\sup_{g,h\in G}|\varphi(gh)-\varphi(g)-\varphi(h)|$$ is finite.

    A quasimorphism is \textit{homogeneous} if it satisfies the additional property
    $$\varphi\left(g^n\right)=n\varphi(g)$$
    for all $g\in G$ and $n\in\mathbb{Z}$.
\end{defn}

Let $X$ be a Gromov-hyperbolic space and $\xi\in\partial X$. In \cite{CCMT15},  Caprace-Cornulier-Monod-Tessera defined a \textit{horokernel} based at $\xi$ to be any accumulation point (in the topology of pointwise convergence) of a sequence of functions
$$X\times X\rightarrow\mathbb{R},\ \ \ \ (x,y)\mapsto d\left(x,x_n\right)-d\left(y,x_n\right),$$
where $\{x_n\}$ is any sequence in $X$ converging to $\xi$. By the Tychonoff theorem, the collection $\mathcal{H}_\xi$ of all horokernels based at $\xi$ is nonempty. 

\begin{proposition}\cite[Proposition~3.7]{CCMT15}\label{Prop: BusemannQM}
    Let $G$ be a group acting by isometries on $X$. Let $\xi\in\partial X$, $h\in\mathcal{H}_\xi$ and $x\in X$. Then the function
$$\beta_\xi\colon \stab_G(\xi)\rightarrow\mathbb{R},\ \ \ \ \beta_\xi(g)=\lim_{n\rightarrow\infty}\frac{1}{n}h\left(x,g^nx\right),$$
is a well-defined homogeneous quasimorphism\footnote{In \cite{CCMT15}, Caprace-Cornulier-Monod-Tessera uses \textit{quasi-characters} to stand for quasimorphisms and \textit{characters} to stand for homomorphisms. For the sake of consistency, we unify these terminologies.}, called {\rm Busemann quasimorphism} of $\stab_G(\xi)$, and is independent of $h\in\mathcal{H}_\xi$ and of $x\in X$. 
\end{proposition}

\begin{remark}
    Actually, \cite[Proposition~3.7]{CCMT15} requires the local compactness of $G$ to get the continuous Busemann quasimorphisms. Since we do not need the Busemann quasimorphism to be continuous in this paper, we drop the local compactness condition.
\end{remark}

The following lemma is a special case of \cite[Lemma~3.8]{CCMT15}.

\begin{lemma}\label{busemann homo q.m.}
    Let $G$ be a group acting isometrically on $X$ and fixing the boundary point $\xi\in \partial X$. Denote by $\beta_\xi$ the corresponding Busemann quasimorphism on $G$. Then $\beta_\xi(g)\ne0$ if and only if $g$ is a hyperbolic isometry of $G\curvearrowright X$.
\end{lemma}

Recall that a group action of $G$ on a Gromov-hyperbolic space $X$ is lineal if $|\Lambda(G)|=2$. For a lineal action $G\curvearrowright X$, we say it is \textit{orientable} if no element of $G$ permutes the two limit points of $\Lambda (G)$. 
Hence, Lemma~\ref{busemann homo q.m.} applies to orientable lineal and focal actions, since there exists a global fixed point in the boundary $\partial X$ for both kinds of actions.

\subsection{Existence of simultaneously hyperbolic elements}\label{subsec: SH}

The main result of this subsection is the following theorem, which gives an affirmative answer to Question~\ref{IntroQue}. Recall from Theorem \ref{Thm: Classification} the classification of group actions on Gromov-hyperbolic spaces.

\begin{theorem}
\label{simulhypiso}
    Let $G$ be a group acting non-elliptically and non-horocyclically on finitely many Gromov-hyperbolic spaces $X_1,\ldots, X_l$. Then the set $$\mathcal{SH}(G)=\{g\in G: g \text{ is simultaneously hyperbolic on each } X_i \text{ for } 1\le i\le l\}$$ is not empty.
\end{theorem}

Since hyperbolic elements are typical contracting elements, Theorem \ref{thm: simul contracting elements} implies the above theorem in the case that each $G$-action on $X_1,\ldots,X_l$ is of general type. To complete the proof, we also need to deal with those lineal or focal actions. We first prove a useful lemma in this process. Recall from Definition \ref{Def: QM} the definition of homogeneous quasimorphisms.

\begin{lemma}
\label{lem: FinQH}
    Let $\beta_1,\cdots,\beta_l$ be non-zero homogeneous quasimorphisms on $G$. Then there exists at least one element $g\in G$ such that $\beta_i(g)\neq 0$ for all $1\le i\le l$. 
\end{lemma}
\begin{proof}
    Denote $A_i=\{g\in G:\beta_i(g)=0\}$ for $1\leq i\leq l$. It suffices to prove that $G-\bigcup_{i=1}^lA_i\ne\varnothing$. We prove by induction on $l$. If $l=1$, then there is nothing to prove since $\beta_1$ is non-zero.

    Suppose our conclusion holds for  $\leq l-1$ and now we consider the case $=l$. By induction hypothesis, we can pick $g\in G-\bigcup_{1\leq i\leq l-1}A_i$ and $h\in G-\bigcup_{2\leq i\leq l}A_i$. WLOG, we can assume that $g\in A_l$ and $h\in A_1$. Otherwise $g$ or $h$ will be in $G-\bigcup_{i=1}^lA_i$ and then we are done. By the homogeneity of $\beta_1,\beta_l$, we have $g^k\in A_l$ and $h^k\in A_1$ for any $k\neq 0\in \mathbb{Z}$.
    
    Denote $\Delta=\max_{1\leq j\leq l} \Delta(\beta_j)$. By the homogeneity of $\beta_1,\cdots, \beta_l$, we can pick $p>0$ such that
    $$\abs{\beta_1(g^p)}=\abs{p\cdot \beta_1(g)}>2\Delta.$$
    Then pick $q>0$ such that 
    $$\abs{\beta_l(h^q)}=\abs{q\cdot \beta_l(h)}>2\Delta$$
    and
    $$\abs{\beta_i(h^q)}=\abs{q\cdot\beta_i(h)}>\abs{\beta_i(g^p)}+2\Delta$$
    for any $2\leq i\leq l-1$.

    Finally, for any $1\leq i\leq l$,
    \begin{align*}
        \abs{\beta_i(g^ph^q)}\geq \abs{\beta_i(g^p)+\beta_i(h^q)}-\Delta\geq \bigg| \abs{\beta_i(h^q)}-\abs{\beta_i(g^p)}\bigg|-\Delta>\Delta>0.
    \end{align*}
    This means that $g^ph^q\in G-\bigcup_{i=1}^lA_i$ and we complete the proof.
\end{proof}

\begin{corollary}
\label{lem: finiteunion q.m.}
    Let $\beta_1,\cdots,\beta_l$ be non-zero homogeneous quasimorphisms on $G$. Denote $A_i=\{g\in G:\beta_i(g)=0\}$ for $1\leq i\leq l$. Then for any finite set $F=\{f_1,\cdots,f_n\}\subset G$, the set
    $$G-\bigcup_{1\leq i\leq l}A_i- \bigcup_{\substack{1\leq i\leq l\\1\leq j\leq n}}A_i\cdot f_j\ne\varnothing.$$ 
\end{corollary}
\begin{proof}
    According to Lemma~\ref{lem: FinQH}, there exists an element $g$ such that $\beta_i(g)\ne 0$ for all $1\leq i\leq l$. Denote $\Delta=\max_{1\leq i\leq l} \Delta(\beta_i)$. Then we can choose $k>0$ such that
    $$\abs{\beta_i(g^k)}=\abs{k\cdot\beta_i(g)}>\max_{1\leq j\leq n}\abs{\beta_i(f_j)}+\Delta$$
    holds for any $1\leq i\leq l$. By the homogeneity of $\beta_1,\cdots, \beta_l$, we know $g^k\notin \bigcup_{i=1}^lA_i$.

    We claim that $g^k\notin \bigcup_{\substack{1\leq i\leq l\\1\leq j\leq n}}A_if_j$. If not, suppose $g^k\in A_if_j$. Then $g^k=a_if_j$ for some $a_i\in A_i$. By definition of $A_i$, $\beta_i(a_i)=0$. Hence we have
    $$\abs{\beta_i(g^k)}=\abs{\beta_i(a_if_j)}\leq \abs{\beta_i(a_i)+\beta_i(f_j)}+\Delta(\beta_i)\leq \abs{\beta_i(f_j)}+\Delta.$$
    This contradicts to our choice of $k$.

    In conclusion, $g^k\in G-\bigcup_{1\leq i\leq l}A_i- \bigcup_{\substack{1\leq i\leq l\\1\leq j\leq n}}A_i\cdot f_j$ and we complete the proof.
\end{proof}

\begin{proposition}
\label{simul hyp m+1-l}
    Let $G$ be a group acting either lineally or focally on Gromov-hyperbolic spaces $X_1,\cdots,X_l$. Then $\sh(G)\neq \varnothing$.
\end{proposition}

\begin{proof}
    WLOG, assume that there exists $0\le n\le l$ such that $G\curvearrowright X_i$ is lineal for $1\leq i \le n$ and $G\curvearrowright X_j$ is focal for $n+1\le j\leq l$. Let $G^\prime$ be a finite-index subgroup of $G$ such that $G^\prime\curvearrowright X_i$ is orientable lineal for $1\le i\le n$. By Lemma~\ref{busemann homo q.m.}, there is a Busemann homogeneous quasimorphism $\beta_i\colon G^\prime\rightarrow\mathbb{R}$ for each $1\le i\le l$ such that $\beta_i(g)\ne0$ if and only if $g$ is hyperbolic on $X_i$. Denote $A_i=\{g\in G':\beta_i(g)=0\}$. By Lemma \ref{lem: FinQH}, $G^\prime-\bigcup_{i=1}^{l}A_i$ is nonempty. This is equivalent to saying that $\sh(G')\neq \emptyset$.
\end{proof}

Now, we are in a position to prove Theorem~\ref{simulhypiso}.

\begin{proof}[Proof of Theorem~\ref{simulhypiso}]
    If there is no action $G\curvearrowright X_i$ is of general type, then the conclusion is given by Proposition~\ref{simul hyp m+1-l}.
    Suppose there exist $1\leq m\leq l$ such that $G\curvearrowright X_i$ is of general type for $1\leq i\leq m$ and $G\curvearrowright X_i$ is either lineal or focal for $m+1\leq i\leq l$. Let $G^\prime$ be a finite-index subgroup of $G$ such that each lineal action $G^\prime\curvearrowright X_i$ is orientable lineal.

    By Corollary~\ref{Cor: SH}, there is a finite set $F\subset \mathcal{SH}(G^\prime;\{1,\cdots,m\})$ satisfies the following property: For any $h\in G^\prime$, there is an $f\in F$ such that $hf\in\mathcal{SH}(G^\prime;\{1,\cdots,m\})$. 
    
    We use the same notations as in proof of Proposition~\ref{simul hyp m+1-l}. Let $\beta_{m+1},\cdots,\beta_l$ be the corresponding Busemann homogeneous quasimorphisms on $G'$ and denote $A_i=\{g\in G': \beta_i(g)=0\}$. 
    By Corollary~\ref{lem: finiteunion q.m.}, the set
    $$G^\prime-\bigcup_{m+1\leq i\leq l}A_i\cup\bigcup_{\substack{m+1\leq i\leq l,\\f\in F}}A_i\cdot f^{-1}$$
    is non-empty and we can choose an element $h$ in it.

    Let $f\in F$ such that $hf\in\mathcal{SH}(G^\prime;\{1,\cdots,m\})$. We claim that $hf\notin\bigcup_{m+1\leq i\leq l}A_i$. Otherwise, $h\in A_if^{-1}$ for some $m+1\le i\le l$ which contradicts to our selection of $h$. By the definitions of $A_i$'s, $hf$ is also hyperbolic on $X_{m+1},\cdots,X_l$. Therefore, $hf\in\mathcal{SH}(G^\prime;\{1,\cdots,l\})\subseteq \mathcal{SH}(G;\{1,\cdots,l\})$.
\end{proof}

In fact, we can also prove
\begin{corollary}
\label{inf ind hyp SH}
    Let $G$ be a group acting non-elliptically and non-horocyclically on Gromov-hyperbolic spaces $X_1,\cdots,X_l$. Suppose there exists $1\le m\le l$ such that $G\curvearrowright X_i$ is of general type for $1\leq i\leq m$. Then there exist infinitely many elements in $\sh(G)$ such that they form an independent subset in $\sh(G;\{1,\cdots,m\})$.
\end{corollary}
\begin{proof}
    By Theorem~\ref{simulhypiso}, we can choose $f\in\mathcal{SH}(G)$ and by Corollary~\ref{Cor: SH}, we can choose an independent set $H=\{h_1,\cdots,h_s\}\subset \mathcal{SH}(G;\{1,\cdots,m\})$ for $s>2m$. By Lemma~\ref{4independent}, there are at least $s-2$ elements in $H$ such that each of them is independent to $f$ on $X_1$. Repeating this process on $X_2$ for these $s-2$ elements in $H$, we get $s-4$ elements in $H$ such that each of them is independent to $f$ on $X_1,X_2$. Repeating this process until on $X_m$, we get at least one element $h$ in $H$ such that it is independent to $f$ on $X_1,\cdots,X_m$. 

    By Lemma~\ref{lem: conjuagte contracting}, there exists $N > 0$ such that for any $n,m > N$, $\{(f^m h^n)^k f (f^m h^n)^{-k} | k \in \mathbb{Z}\}$ is an infinite independent subset in $\sh(G;\{1,\cdots,m\})$. Note that they are all conjugations of $f$, so they also belong to $\sh(G)$.
\end{proof}

\subsection{Positive density of $\mathcal{SH}(G)$}\label{subsec: PosiDens}

In the last part of this paper, we are going to prove Theorem~\ref{IntroThm: SH pos den}. The proof strategy is the same as in subsection~\ref{subsec: positive density sc}. We first prove the following lemma which is an analogue of \cite[Theorem 2]{CW18}.


\begin{lemma}
\label{lem: positive density}
    Let $G$ be a group acting non-elliptically and non-horocyclically on finitely many Gromov-hyperbolic spaces $X_1,\ldots, X_l$. Then there exists a finite set $F\subset G$ with the following property. For any $g\in G$, there exists $f\in F$ such that $gf\in \mathcal{SH}(G)$.
\end{lemma}

If each $G$-action on $X_1,\ldots,X_l$ is of general type, then Lemma~\ref{generalized extension lemma} implies that

\begin{corollary}\label{Cor: SH}
    Suppose that a group $G$ acts by general type actions on Gromov-hyperbolic spaces $X_1,\cdots,X_l$. Then there is a finite set $F\subset \sh(G)$ with the following property. For any $g\in G$, there is an $f\in F$ such that $gf\in \sh(G)$.
\end{corollary}

We now deal with another special case.

\begin{lemma}\label{lem: GEL for NoGT}
    Let $G$ be a group acting either lineally or focally on Gromov-hyperbolic spaces $X_1,\ldots, X_l$. Then there exists a finite set $F\subset \mathcal{SH}(G)$ with the following property. For any $g\in G$, there exists $f\in F$ such that $gf\in \mathcal{SH}(G)$.
\end{lemma}

\begin{proof}
    Let $G^\prime$ be a finite-index subgroup of $G$ such that each lineal action $G^\prime\curvearrowright X_i$ is orientable lineal. We use the same notations as in proof of Proposition~\ref{simul hyp m+1-l}. Let $\beta_1,\cdots,\beta_l$ be the corresponding Busemann homogeneous quasimorphisms on $G'$ and denote $A_i=\{g\in G':\beta_i(g)=0\}$. Let $A\coloneqq\bigcup_{i=1}^l A_i$ and thus $\mathcal{SH}(G^\prime)=G^\prime-A$.

    By Proposition~\ref{simul hyp m+1-l}, $\sh(G^\prime)\neq \varnothing$ and we can pick $h\in G^\prime-A$. By the homogeneity of $\beta_1,\cdots,\beta_l$, $h^k\notin A$ for any $k\ne 0$. Then we pick $k>0$ such that
    $$|\beta_i(h^k)|=k|\beta_i(h)|>2\Delta$$
    for any $1\leq i\leq l$, where $\Delta=\max_{1\leq i\leq l}\Delta(\beta_i)$.
    Denote $F=\{h^{k},\cdots,h^{(l+1)k}\}$. By constructions, $F\subset\mathcal{SH}(G^\prime)\subset\sh(G)$.

    We claim that $F$ is the required set. If not, suppose there exists $g\in G$ such that $gh^{k},\cdots,gh^{(l+1)k}\notin\mathcal{SH}(G)$. This implies that $gh^{k},\cdots,gh^{(l+1)k}\in A=\bigcup_{i=1}^l A_i$. By the Pigeonhole Principle, there are $1\le p<q\le l+1$ and $1\leq i\leq l$ such that $gh^{pk},gh^{qk}\in A_i$. Then we have $\abs{\beta_i(gh^{pk})}=\abs{\beta_i(gh^{qk})}=0$ and thus $\abs{\beta_i(g)+\beta_i(h^{pk})}\leq \Delta$ and $\abs{\beta_i(g)+\beta_i(h^{qk})}\leq \Delta$. By the homogeneity of $\beta_i$ and triangle inequality, one gets that
    \begin{align*}
        (q-p)k|\beta_i(h)|=|\beta_i(h^{qk})-\beta_i(h^{pk})|\le |\beta_i(h^{qk})+\beta_i(g)-(\beta_i(g)+\beta_i(h^{pk}))|\le 2\Delta,
    \end{align*}
    which contradicts to the choice of $k$.
\end{proof}

We are in  a position to prove Lemma~\ref{lem: positive density}. The proof is similar to the proof of Lemma~\ref{generalized extension lemma}. 

\begin{proof}[Proof of Lemma~\ref{lem: positive density}]
    If each action $X\curvearrowright X_k$ is either lineal or focal for $1\le k\le l$, then the conclusion follows from Lemma~\ref{lem: GEL for NoGT}.
    Suppose there exists $1\leq m\leq l$ such that $G\curvearrowright X_i$ is of general type for $1\leq i\leq m$ and $G\curvearrowright X_i$ is either lineal or focal for $m+1\leq i\leq l$. Let $G^\prime$ be a finite-index subgroup of $G$ such that each lineal action $G^\prime\curvearrowright X_i$ is orientable lineal.  By using the same notations as in the proof of Lemma~\ref{lem: GEL for NoGT}, we denote by $\beta_{m+1},\cdots,\beta_l$ the corresponding Busemann homogeneous quasimorphisms on $G'$ and $A_i=\{g\in G': \beta_i(g)=0\}$. Let $A=\bigcup_{i=m+1}^l A_i$ and $\Delta=\max_{m+1\leq i\leq l}\Delta(\beta_i)$. Note that $\sh(G^\prime;\{m+1,\cdots,l\})=G^\prime-A$. 
    
    By Corollary~\ref{inf ind hyp SH}, we can choose a subset $\{h_1,\cdots,h_s\}\subset\mathcal{SH}(G^\prime;\{1,\cdots,l\})$ for $s>2l+1$ such that it is an independent set in $\sh(G^\prime;\{1,\cdots,m\})$. Denote $H=\left\{h_1,\cdots,h_s\right\}$. Fix a basepoint $o_k\in X_k$ for $1\le k\le m$. By the Extension Lemma (cf. Lemma~\ref{extensionlemma}), there exists a uniform constant $D>0$ with the following property: For any $1\leq k\leq m$, for any three elements $h_a,h_b,h_c\in H$, and any three integers $M_a,M_b,M_c$ satisfying $d\left(o_k,{h_j}^{M_j}o_k\right)>D$ for all $j\in\{a,b,c\}$, and any $g\in G^\prime$ satisfying $d\left(o_k,go_k\right)>D$, there exists $h\in\left\{ {h_a}^{M_a},{h_b}^{M_b},{h_c}^{M_c} \right\}$ such that $gh,hg$ are hyperbolic on $X_k$. 

    For each $1\leq i\leq s$, we can pick $n_i\gg 0$ such that 
    \begin{enumerate}
        \item $\min_{1\leq k\leq m}d(o_k,h_i^{n_i}o_k)\ge 2D$.
        \item $\abs{\beta_k(h_i^{n_i})}=\abs{n_i\cdot \beta_k(h_i)}>2\Delta$ for each $m+1\leq k\leq l$, $1\leq i\leq s$.
    \end{enumerate}

    Let $f_i=h_i^{n_i}$ for $1\le i\le s$. 
    Denote $L_k=d(o_k,f_1o_k)$ for $1\leq k\leq m$ for short. Denote $M=l-m+1$.
    
    Then pick $j_0=0<j_1<\cdots<j_{m}$ inductively such that 
    $$d(o_k,(f_1)^{j_r}o_k)>2D+j_{r-1}L_k$$ for any $1\leq k,r\leq m$.
    
    Let 
    $$F^\prime=\{f_1,\cdots,f_s,f_1^2,\cdots,f_s^2,\cdots,f_1^M,\cdots,f_s^M\}$$
    and 
    $$F=\bigcup_{r=0}^{m} f_1^{j_r}\cdot F^\prime$$

    \begin{claim}
    For any $g\in G$ and $1\le i\le s$, there exists $1\le M_i\le M$ such that $gf_i^{M_i}\in \sh(G^\prime; \{m+1,\cdots,l\})$.
    \end{claim}
    \begin{proof}[Proof of Claim]
        Recall that $A=\bigcup_{j=m+1}^l A_j$ and $\sh(G^\prime;\{m+1,\cdots,l\})=G^\prime-A$. It suffices to show that $\{gf_i, gf_i^2, \cdots, gf_i^{l-m+1}\}$ is not contained in $A$. If not, by the Pigeonhole Principle, there exist $m+1\le j\le l$ and $1\le p<q\le l-m+1$ such that $gf_i^p,gf_i^q\in A_j$. Then we have $\abs{\beta_j(gf_i^p)}=\abs{\beta_j(gf_i^q)}=0$ and thus $\abs{\beta_j(g)+\beta_j(f_i^p)}\leq \Delta$ and $\abs{\beta_j(g)+\beta_j(f_i^q)}\leq \Delta$. By the homogeneity of $\beta_j$ and triangle inequality, one gets that
    \begin{align*}
        (q-p)|\beta_j(f_i)|=|\beta_j(f_i^{q})-\beta_j(f_i^{p})|\le |\beta_j(f_i^{q})+\beta_j(g)-(\beta_j(g)+\beta_j(f_i^{p}))|\le 2\Delta,
    \end{align*}
    which contradicts to the choice of $f_i$. 
    \end{proof}

    The remaining proof can be completed by similar arguments in the proof of Lemma~\ref{generalized extension lemma}. For any $g\in G$, the Claim in the proof of Lemma~\ref{generalized extension lemma} shows that there exists an $h\in\{gf_1^{j_0},gf_1^{j_1},gf_1^{j_2},\cdots,gf_1^{j_{m}}\}$ such that $d(o_k,ho_k)>D$ for any $1\leq k\leq m$. For such $h$, let $f_1^{M_1},\cdots,f_s^{M_s}$ be given by the above Claim such that $hf_i^{M_i}\in\sh(G';\{m+1,\cdots,l\})$ for any $1\leq i\leq s$. By using \hyperref[SC construction]{SC construction} for $h$ on $X_1,\cdots,X_m$ with $f_1^{M_1},\cdots,f_s^{M_s}$, we get an element $hf_i^{M_i}\in\sh(G';\{1,\cdots,m\})$. We conclude that $hf_i^{M_i}\in\sh(G')$, which means that there exists an element $f_1^{j_r}f_i^{M_i}\in F$ for some $0\leq r\leq m$ and some $1\leq i\leq s$ such that $gf_1^{j_r}f_i^{M_i}\in \sh(G')\subset \sh(G)$.
\end{proof}

\begin{corollary}[Positive density of simultaneously hyperbolic elements]\label{PosDes sh}
    Let $G$ be a group acting non-elliptically and non-horocyclically on finitely many Gromov-hyperbolic spaces $X_1,\ldots, X_l$. If $G$ is finitely generated by $S$, then there exists a constant $c=c(S)\in (0,1)$ such that $$\frac{\sharp  (S^{\le n}\cap \mathcal{SH}(G))}{\sharp S^{\le n}}>c$$ for any sufficiently large $n$.
\end{corollary}
\begin{proof}
    Combining Lemma~\ref{lem: PosDes} and Lemma~\ref{lem: positive density} completes the proof.
\end{proof}

Finally, we give an example to respond Remark~\ref{rem: not generic}~(2).

\begin{example}\label{Exa: DesGAP}
    Let $G=F_2\times F_3$ and $S_1, S_2$ be free bases of $F_2,F_3$, respectively. Then $G$ is generated by $S\coloneqq\{(x,1),(1,y):x\in S_1,y\in S_2\}$. Now we consider the diagonal action of $G$ on the product of two locally finite regular trees $\mathcal{G}(F_2,S_1)$ and $\mathcal{G}(F_3,S_2)$. 
    
    Some direct computations show that $\sharp S_1^{\leq n}=2\cdot 3^n-1$ and $\sharp S_2^{\leq n}=\frac{3}{2}5^n-\frac{1}{2}$. Then
    \begin{align*}
        \sharp S^{\leq n}&=\sum_{k=0}^n\sum_{i+j=k}\sharp S_1^{\leq i}\cdot\sharp S_2^{\leq j}\\
        &=\sum_{k=0}^n\left( 
        \frac{9}{8}5^{k+1}-2\cdot 3^{k+1}+\frac{4k+11}{8} \right)\\
        &=\frac{45}{32}\left(5^{n+1}-1\right)+3\left(-3^{n+1}+1\right)+\frac{(n+1)(2n+11)}{8}.
    \end{align*}
    Next, we count the number of non-simultaneously hyperbolic in $S^{\leq n}$. Note that $\sh(G)=\{(x,y):x\ne 1\in F_2,y\ne 1\in F_3\}$. Hence, the set of non-simultaneously hyperbolic elements is $\{(x,1),(1,y):x\in F_2,y\in F_3\}$. This implies that
    $$\sharp\left( 
    S^{\leq n}-\sh(G) \right)=\sharp S_1^{\leq n}+\sharp S_2^{\leq n}-1=\frac{3}{2}5^n+2\cdot 3^n-\frac{5}{2}.$$
    Therefore, we have 
    $$\lim_{n\rightarrow\infty}\frac{\sharp\left( 
    S^{\leq n}-\sh(G) \right)}{\sharp S^{\leq n}}=\lim_{n\rightarrow\infty}\frac{\frac{3}{2}5^n+2\cdot 3^n-\frac{5}{2}}{\frac{45}{32}\left(5^{n+1}-1\right)+3\left(-3^{n+1}+1\right)+\frac{(n+1)(2n+11)}{8}}=\frac{48}{225}>0.$$
    Hence, there exists a constant $c'\in (0,1)$ such that 
    $$\frac{\sharp\left( 
    S^{\leq n}\cap\sh(G) \right)}{\sharp S^{\leq n}}\leq c'<1$$
    for any sufficiently large $n$.

    Note that in this example, $\sh(G)=\sc(G)$. Therefore this example answers Remark~\ref{rem: not generic}~(2).
\end{example}

\bibliographystyle{amsplain}   
\bibliography{Reference}
\end{document}